\documentclass[12pt]{article}

\usepackage{amsmath,amsthm,amssymb,graphicx,tikz,mathtools,enumitem}
\usepackage{tikz-cd}
\usepackage[margin=1in]{geometry}
\usepackage{mathrsfs}
\usepackage[all]{xy}
\usepackage{amsfonts}
\usepackage{latexsym}
\usepackage{subfigure}
\usepackage{fullpage}
\usepackage{float}
\usepackage{xypic}
\usepackage{bbm}

\usepackage[percent]{overpic}

\DeclareMathOperator{\Kh}{Kh}
\DeclareMathOperator{\Lee}{Lee}
\DeclareMathOperator{\Mod}{Mod}
\DeclareMathOperator{\CFK}{CFK}

\newcommand{\Z}{\mathbb{Z}}

\newcommand{\fraks}{\mathfrak{s}}

\newtheorem{thm}{Theorem}[section]
\newtheorem{prop}[thm]{Proposition}
\newtheorem{lem}[thm]{Lemma}
\newtheorem{cor}[thm]{Corollary}

\theoremstyle{definition}
\newtheorem{rem}[thm]{Remark}

\newtheorem{example}[thm]{Example}

\usepackage{hyperref}
\hypersetup{pdftex,colorlinks=true,allcolors=blue}

\begin{document}
\title{Annular Rasmussen invariants: Properties and 3-braid classification}
\author{Gage Martin}
\date{}
\maketitle

\begin{abstract}
We prove that for a fixed braid index there are only finitely many possible shapes of the annular Rasmussen $d_t$ invariant of braid closures. Applying the same perspective to the knot Floer invariant $\Upsilon_K(t)$, we show that for a fixed concordance genus of $K$ there are only finitely many possibilities for $\Upsilon_K(t)$. Focusing on the case of 3-braids, we compute the Rasmussen $s$ invariant and the annular Rasmussen $d_t$ invariant of all 3-braid closures. As a corollary, we show that the vanishing/non-vanishing of the $\psi$ invariant is entirely determined by the $s$ invariant and the self-linking number.
\end{abstract}

\section{Introduction}

In~\cite{khovanov_categorification_2000}, Khovanov defined a bigraded homology theory $\Kh^{i,j}(L)$ associated to an oriented link $L\subseteq S^3$. Later in~\cite{lee_endomorphism_2005}, Lee defined a homology theory $\Lee(L)$ by adding additional differentials to the Khovanov chain complex. Lee also showed that the total rank of $\Lee(L)$ is $2^{\vert L \vert}$ where $\vert L \vert$ is the number of components of $L$.

For a knot $K$, J.~Rasmussen used a $\mathbb{Z}$ filtration on Lee homology to define an invariant $s(K)$~\cite{rasmussen_khovanov_2010}. His invariant gives a lower bound on the smooth 4-ball genus of a knot $K$ and is strong enough to give a combinatorial reproof of the Milnor conjecture about the smooth 4-ball genus of torus knots~\cite{john_w_milnor_singular_1968}. The definition of the $s$ invariant was later extended to oriented links by Beliakova and Wehrli~\cite{beliakova_categorification_2008}. Pardon gives a slightly different extension of the Rasmussen invariant to oriented links in~\cite{pardon_link_2012} but in the present paper we use the extension by Beliakova and Wehrli. 

In a slightly different direction, Asaeda, Przytycki, and Sikora~\cite{asaeda_categorification_2004} and L.~Roberts~\cite{roberts_knot_2013} define a version of Khovanov homology called annular Khovanov homology for oriented links $L$ embedded in a thickened annulus $A \times I$ which is triply graded. Additionally if the thickened annulus is embedded in $S^3$ so that is it unknotted, then the additional grading on annular Khovanov homology of $L$ induces a $\mathbb{Z}$ filtration on the standard Khovanov homology of $L$. Annular Khovanov homology detects the trivial braid closure~\cite{baldwin_categorified_2015} and detects some non-conjugate braids related by exchange moves~\cite{hubbard_sutured_2017}.

Then in~\cite{grigsby_annular_2017}, Grigsby, A.~Licata, and Wehrli combine the ideas above and show that for an oriented annular link $L \subset A \times I \subset S^3$ that the Lee homology of $L$ is $\mathbb{Z} \oplus \mathbb{Z}$ filtered. From this data, using ideas of Ozsv\'{a}th, Stipsicz, and Szab\'{o}~\cite{ozsvath_concordance_2017}, as reinterpreted by Livingston~\cite{livingston_notes_2017}, they construct a piecewise linear function $d_t(L)$ called the annular Rasmussen invariant. At $t = 0 $ the $d_t$ invariant recovers the $s$ invariant of $L$ by $s(L) -1 = d_t(L)$ and additionally  Grigsby, A.~Licata, and Wehrli show that for braid closures $d_t(\widehat{\beta})$ can be used to show that $\widehat{\beta}$ is right-veering and also to show that $\widehat{\beta}$ is not quasipositive.

In the present paper we investigate the $\mathbb{Z} \oplus \mathbb{Z}$ filtered Khovanov-Lee complex of braid closures and use its algebraic structure to obtain strong restrictions on the annular Rasmussen invariant.

\newtheorem*{dtn}{Theorem~\ref{dtn}}
\begin{dtn}
For an $n$-braid $\beta$, there are only finitely many possible shapes that $d_t( \widehat{\beta})$ can take and there is a method for enumerating all the possibilities.
\end{dtn}

\begin{rem}
For 3-braids, the number of possible shapes of $d_t( \widehat{\beta})$ is three. For 4-braids, the number is seven and for 5-braids there are 18 possible shapes.
\end{rem}

Applying some of the same techniques to the $\mathbb{Z} \oplus \mathbb{Z}$ filtered knot Floer complex gives restrictions on the invariant $\Upsilon_K(t)$ originally defined by Ozsv\'{a}th, Stipsicz, and Szab\'{o} in~\cite{ozsvath_concordance_2017}.
\newtheorem*{UpsilonFin}{Theorem~\ref{UpsilonFinite}}
\begin{UpsilonFin}
For a fixed concordance genus $c$ there are finitely many possibilities for $\Upsilon_K(t)$ for any $K$ of concordance genus $c$ and a method for enumerating all the possibilities.
\end{UpsilonFin}
\begin{rem}
For concordance genus 1, the number of possibilities for $\Upsilon_K(t)$ is five. For concordance genus 2, there are over 50 possibilities for $\Upsilon_K(t)$.
\end{rem}

Because the value of $d_t(\widehat{\beta})$ at $t = 0$ is $s(\widehat{\beta}) - 1$, one may hope that Theorem~\ref{dtn} provides a new upper bound on the $s$ invariant of braid closures. Unfortunately, the upper bound provided by Theorem~\ref{dtn} for a braid $\beta$ with braid index $n$ is $s(\widehat{\beta}) \leq w(\widehat{\beta}) + n - 2$ which is never better than the bounds coming from~\cite{lobb_computable_2011} when Lobb's upper bound $U(D)$ is computed for the diagram $D = \widehat{\beta}$.
 
 Applying this perspective on $d_t( \widehat{\beta})$ to 3-braids we get that the $d_t$ invariant of any 3-braid closure $\widehat{\beta}$ depends only on the writhe $w(\widehat{\beta})$ of the braid closure and the Rasmussen invariant $s(\widehat{\beta})$ of the closure.
 
\newtheorem*{dt3}{Theorem~\ref{dt3}}
\begin{dt3}
When $\beta$ is a 3-braid, then for $t$ between $0$ and $1$ one of the following holds $d_t( \widehat{\beta}) = w(\beta) - 3 + 3t $, $d_t( \widehat{\beta}) = w(\beta) - 1 + t$ or $d_t( \widehat{\beta}) = w(\beta) +1- t$.
\end{dt3}
In other words, if $\beta$ is a 3-braid then $d_t( \widehat{\beta})$ is entirely determined by the $s$ invariant and the writhe.

With Theorem~\ref{dt3} in mind, all that is needed to compute the $d_t$ invariants of 3-braid closures is to compute the $s$ invariant of all 3-braid closures. Focusing on 3-braid closures allows us to use Murasugi's classification of 3-braids up to conjugacy. By understanding how the $s$ invariant changes under adding crossings we are able to compute the $s$ invariant of all 3-braid closures.

\begin{thm}
The value of the $s$ invariant can be read off from Murasugi's classification for every 3-braid closure.
\end{thm}

A more detailed statement of value of the $s$ invariant of 3-braid closures can be found in Theorems~\ref{slope3},~\ref{slope-1}, and~\ref{slope1}. A discussion of why Murasugi's classification is used here instead of a more modern approach to the conjugacy problem for 3-braids can be found in Remark~\ref{Xu}.

Because having non-zero $s$ invariant obstructs sliceness, it is a natural question to ask if there are some 3-braid closures that were previously unknown if they were slice and our computation shows they have non-zero $s$ invariant. However the only 3-braids where it is not know if they are slice are known to have finite concordance order and so their $s$ invariants are necessarily zero~\cite{lisca_3-braid_2017}.

Using braid foliations, Birman and Menasco completely classify which links are closures of 3-braids and in particular they find two infinite families of non-conjugate 3-braids whose closures are isotopic as links~\cite{birman_studying_1993}. Knowing that the $d_t$ invariants of a braid closure are invariants of the conjugacy class of of the braid, one may ask if the $d_t$ invariants of 3-braid closures can detect these families of non-conjugate 3-braids. However, the $d_t$ invariants of 3-braid closures only detect the writhe of the braid closure and an invariant of the isotopy class of the braid closure, the $s$ invariant, so they can not distinguish the non-conjugate 3-braids.

By comparing the value of the $s$ invariant of 3-braid closures with the vanishing/non-vanishing of $\psi$ defined in~\cite{plamenevskaya_transverse_2006} we obtain the following corollary of Theorem~\ref{slope3}.
\newtheorem*{NotEffective}{Corollary \ref{NotEffective}}

\begin{NotEffective}
The invariant $\psi$ is not effective for 3-braid closures. In particular, the vanishing/non-vanishing of $\psi$ for 3-braids is determined by the $s$ invariant and the self-linking number.
\end{NotEffective}

It is currently an open question if $\psi$ is an effective transverse invariant and relatedly it is also unknown if the vanishing/non-vanishing of $\psi$ is always determined only by the $s$ invariant and the self-linking number.

The organization of the paper is as follows. In Section~\ref{background}, we review some basic facts about 3-braids and braid closure invariants from Khovanov homology. In Section~\ref{additionalproperties}, we investigate new properties of the annular Rasmussen $d_t$ invariants and prove Theorem~\ref{dtn}. In Section~\ref{s3braids}, we compute the Rasmussen $s$ invariant for 3-braid closures and prove Theorems~\ref{slope3},~\ref{slope-1}, and~\ref{slope1} and in Subsection~\ref{Upsilon} we apply some of the same techniques to prove Theorem~\ref{UpsilonFinite}. In Section~\ref{compare}, we compare our computations of the $s$ invariant of 3-braid closures to the computations of other categorified invariants of 3-braid closures. Section~\ref{computation} contains an explicit computation of the $s$ invariant of a single 3-braid closure that was not computable via the methods used in Section~\ref{s3braids}.

\subsection*{Acknowledgements}

The questions addressed in this paper were inspired by a question Joan Birman asked Eli Grigsby and Anthony Licata about the relationship between braid conjugacy invariants from annular Khovanov-Lee homology and classical solutions to the braid conjugacy problem. The author would like to thank John Baldwin, Eli Grigsby, Scott Morrison and Melissa Zhang for helpful conversations. The author would also like to thank Joan Birman for helpful comments on a draft of this paper.
 
\section{Background}\label{background}

\subsection{3-braids}

\begin{figure}

  \centering
    \begin{overpic}[width=0.05\textwidth]{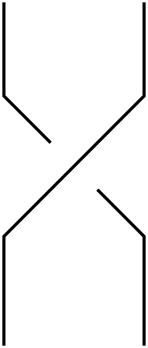}
    \put(-4,-20){$i$}
    \put(20,-20){$i+1$}
    \end{overpic}
      \caption{A local diagram for the generator $\sigma_i$}\label{sigmai}
\end{figure}
\begin{figure}

  \centering
    \includegraphics[width=0.3\textwidth]{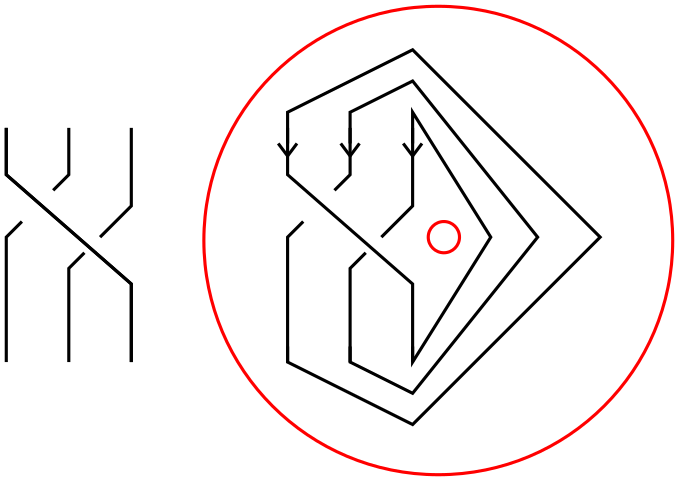}
      \caption{A braid and the annular closure of the braid equipped with the braid orientation}\label{annularclosure}
\end{figure}

A braid on $n$ strands is a proper embedding of $n$ disjoint copies of the interval $I = [0,1]$ into $D^2 \times [0,1]$ where each interval intersects every slice $D^2 \times t \in [0,1]$ in exactly one point up to isotopy through such embeddings. The collection of all braids on $n$ strands form a group $B_n$ with multiplication given by stacking braids. The group $B_n$ has the Artin presentation with generators $\sigma_i$ for $1 \leq i \leq n-1$. Diagrammatically, the generator $\sigma_i$ is the crossing in Figure~\ref{sigmai} between the $i$-th and $i+1$-st strands. The group $B_n$ can be identified with $\Mod(D_n)$ the mapping class group of the disk with $n$ punctures. See~\cite{birman_braids:_2005} for more about braids.

Given a braid $\beta$, there is a link $ \widehat{\beta}$ in the thickened annulus associated to $\beta$ called the annular closure of $\beta$. The braid closure $ \widehat{\beta}$ has an orientation called the braid orientation induced on it by the braid $\beta$ where all the strands of $\beta$ are oriented to flow from the top of the braid to the bottom. See Figure~\ref{annularclosure} for a diagram of the annular closure with the braid orientation. Two oriented annular braid closures $\hat{\beta_1}$ and $\hat{\beta_2}$ represent the same oriented annular link if and only if $\beta_1$ and $\beta_2$ are conjugate braids. All braid closures are assumed to be equipped with the braid orientation. 

A braid $\beta$ is said to be quasipositive if it can be expressed as a product of conjugates of positive generators, that is if $\beta = \prod w_i \sigma_{j_i} w_i^{-1}$ for some words $w_i\in B_n$. A braid $\beta$ is said to be right-veering if it sends every arc $\gamma$ running from $\partial D$ to one of the marked points to the right. See Section~2.2 of~\cite{plamenevskaya_braid_2018} for more details and a precise definition. All quasipositive braids are right-veering, but there are examples of right-veering braids which are not quasipositive. Some examples of this can be found, e.g.~in~\cite{plamenevskaya_braid_2018}. 

For each braid group $B_n$ there is a group homomorphism $w:B_n \to \mathbb{Z}$ given by $\sigma_i \to 1$ for every $i$. The image $w(\beta)$ is an invariant of the annular closure $ \widehat{\beta}$ called the writhe $w( \widehat{\beta})$. The mirror of a braid $\beta$ is the image of $\beta$ under the homomorphism $m: B_n \to B_n$ given by $\sigma_i \to \sigma_i^{-1}$ for every $i$. The annular closure $
\widehat{m(\beta)}$ is the topological mirror of the annular closure $\widehat{\beta}$. Note that $w(\widehat{m(\beta)}) = -w(\widehat{\beta})$.

Murasugi classified all $3$-braids up to conjugation as having exactly one of the following three forms.

\begin{lem}[Murasugi's classification of 3-braids \cite{murasugi_closed_1974}]

Every $3$-braid is conjugate to one of the following braids:
\begin{enumerate}
\item $h^d\sigma_1 \sigma_2^{-a_1} \cdots \sigma_1 \sigma_2^{-a_n}$ with $a_i \geq 0$ and some $a_i > 0$.
\item $h^d \sigma_2^m$ with $m \in \Z$.
\item $h^d \sigma_1^m \sigma_2^{-1}$ with $m \in \{ -1 ,-2, -3\}$.
\end{enumerate}
where $h = (\sigma_1 \sigma_2)^3$.
\end{lem}

In what follows, we will refer to the three different forms for a 3-braid listed above as families 1, 2, and 3 respectively.

\begin{rem}\label{Xu}
Since the time since Murasugi published his classification of 3-braids, much more has been written about the structure of 3-braids. Specifically, work of Xu provided a new solution to the conjugacy problem for 3-braids and a solution to the shortest word problem for 3-braids~\cite{xu_genus_1992}. Xu's approach also generalizes to solving the conjugacy problem for $n$-braids. The reason why we use Murasugi's approach in what follows is because Murasugi's classification is connected to properties of the of $\beta \in \Mod(D_3)$ and it was previously known that the $d_t$ invariants of braid closures are connected to certain properties of the conjugacy class of $\beta$ viewed as a mapping class, for example if the conjugacy class is right-veering. A summary of Xu's solution to the conjugacy problem for 3-braids and other more recent knowledge about 3-braids can be found in~\cite{birman_note_2008}.
\end{rem}

\subsection{Braid closure invariants from Khovanov homology}

Given a diagram of an annular braid closure $ \widehat{\beta}$, we will build two $\mathbb{Z}\oplus \mathbb{Z}$ filtered chain complexes using a constructions of Khovanov~\cite{khovanov_categorification_2000} and Lee~\cite{lee_endomorphism_2005} along with ideas from J.~Rasmussen~\cite{rasmussen_khovanov_2010}, Asaeda-Przytycki-Sikora~\cite{asaeda_categorification_2004}, and L.~Roberts~\cite{roberts_knot_2013}. Given a braid closure $ \widehat{\beta}$, we can form the cube of resolutions of $ \widehat{\beta}$ as in Section~4.2 of~\cite{khovanov_categorification_2000}. This cube of resolutions is then assigned a triply graded vector space $C^{i,j,k}( \widehat{\beta})$. The $i,j,$ and $k$ gradings are the homological, quantum, and annular gradings respectively. This triply graded vector space can be equipped with a differentials $\partial$ and $\Phi$ so that the homology with respect to $\partial$ and also with respect to $\partial + \Phi$ are called the Khovanov homology $\Kh^{i,j}( \widehat{\beta})$ and the Lee homology $\Lee^i( \widehat{\beta})$ respectively. Both homologies are invariants of the oriented link $L \subseteq S^3$ represented by $ \widehat{\beta}$. Because $\partial$ preserves the $j$ grading, Khovanov homology is bi-graded with $i$ and $j$ gradings and filtered by $j - 2k$.  Lee homology is $i$ graded and $\mathbb{Z} \oplus\mathbb{Z}$ filtered by $j$ and $j - 2k$. The $\mathbb{Z} \oplus\mathbb{Z}$ filtered chain homotopy type of Lee homology is an invariant of the annular closure of $ \widehat{\beta}$. The chain complex for Lee homology decomposes into two subcomplexes $L_1$ and $L_2$ depending on the quantum grading of vectors mod~4.

The vector space $C^{i,j,k}( \widehat{\beta})$ has as a basis distinguished generators, each of which is homogeneous with respect to each grading. A complete resolution of $ \widehat{\beta}$ is a choice of smoothing for each crossing of $ \widehat{\beta}$. A complete resolution consists of a collection of circles in the annulus and some wrap around the puncture of the annulus while others will not. Associated to each vertex of the cube of resolutions of $ \widehat{\beta}$ is a complete resolution. Generators of $C^{i,j,k}( \widehat{\beta})$ are labelings of each circle of this complete resolution with either a $+$ sign or a $-$ sign. Each circle with a $+$ increases the quantum grading by one and each circle with a $-$ decreases the quantum grading by one. The quantum grading is then shifted depending on the resolution that is labeled. For the annular grading, only circles that wrap around the annulus are counted, among these circles those that are labeled with a $+$ increases the annular grading by one and each circle with a $-$ decreases the annular grading by one. For an $n$ braid, the annular grading of any generator is in the set $\{ -n , -n + 2, \ldots , n - 2, n \}$. The homological grading is defined in such a way that if $ \widehat{\beta}$ has braid orientation $o$ then the oriented resolution coming from $o$ sits in homological grading $0$.

Lee showed that if $ \widehat{\beta}$ is a link with $m$ components then $\Lee( \widehat{\beta})$ has dimension $2^m$ and for each possible orientation $o$ of $ \widehat{\beta}$ constructed an explicit cycle $\mathfrak{s}_o$ so that the homology classes of these cycles generate $\Lee( \widehat{\beta})$~\cite{lee_endomorphism_2005}. The class $\mathfrak{s}_0$ is a linear combination of distinguished generators in the oriented resolution coming from $o$. Each possible labeling of the resolution appears in the linear combination with coefficient $\pm 1$ depending on the orientation and which circles are labeled with a $+$ and which are labeled with $-$. The reader can refer to~\cite{rasmussen_khovanov_2010} for a detailed description of how to construct the cycle $\mathfrak{s}_o$. 

Using Lee homology, J.~Rasmussen defined the $s$ invariant for knots in $S^3$ and his definition was extended to oriented links in $S^3$ by Beliakova and Wehrli~\cite{rasmussen_khovanov_2010}~\cite{beliakova_categorification_2008}. We recall some basic properties of this invariant. For a braid closure $ \widehat{\beta}$ with its braid orientation $o$ there is a non-zero homology class $[\mathfrak{s}_o]$ in $\Lee^0( \widehat{\beta})$. It follows immediately from the definition~(Definition~3.4~of~\cite{rasmussen_khovanov_2010}) that $s( \widehat{\beta}) = \text{max}\{ gr_{j}(x) \mid [x] = [\mathfrak{s}_o] \} +1$ where $gr_{j}$ is the filtered $j$ degree. The $s$ invariant gives bounds on the four ball genus of knots and links and is a group homomorphism from the smooth knot concordance group to $\mathbb{Z}$~\cite{rasmussen_khovanov_2010}~\cite{beliakova_categorification_2008}. 
\begin{prop}[Proposition~3.9 of~\cite{rasmussen_khovanov_2010}]\label{mirror} If $K$ is a knot then $s(m(K)) = - s(K)$.
\end{prop}
\begin{prop}[Lemma~3.5 of~\cite{rasmussen_khovanov_2010}]\label{subcomplex} For an oriented link $L$, the homology class $[\mathfrak{s}_o]$ is non-zero in both subcomplexes of $\Lee(L)$.
\end{prop}

By considering the additional $j-2k$ filtration on Lee homology, Grigsby, A.~Licata, and Wehrli extended the $s$ invariant to a family of annular braid closure invariants $d_t$ one for each $t$ in the interval $[0,2]$, this family of invariants is also called the annular Rasumssen invariants~\cite{grigsby_annular_2017}. For each $t \in [0,2]$ the Lee homology has a filtration $j_t = j - t k $ so for each $t$ it is possible to define $d_t( \widehat{\beta})= \text{max}\{ j_t(x) \mid [x] = [\mathfrak{s}_o] \}$. From the definition it is immediate that $d_0( \widehat{\beta}) = s( \widehat{\beta}) - 1$. After defining the invariant, they show that it has a symmetry $d_{1-t}( \widehat{\beta}) = d_{1+ t}( \widehat{\beta})$~\cite{grigsby_annular_2017}. Because of the symmetry, in the rest of this paper we restrict to the interval $[0,1]$ without losing information about the invariant $d_t$. We recall key properties of the $d_t$ invariant needed for the present work.
\begin{prop}[Theorem~1 of~\cite{grigsby_annular_2017}] The $d_t$ invariant is a piecewise linear function on $[0,1]$ and the right-hand slope of $d_t( \widehat{\beta})$, called $m_t( \widehat{\beta})$, at any $t \in [0,1]$ is the negative $k$ grading of the cycle~$x$ with $[x] = [\mathfrak{s}_o]$ and $j_t(x) = d_t( \widehat{\beta})$.\end{prop}
\begin{prop}[Theorem~3 of~\cite{grigsby_annular_2017}] Let $ \widehat{\beta}$ be a braid closure with writhe $w( \widehat{\beta})$, then $d_1( \widehat{\beta}) = w( \widehat{\beta})$.\end{prop}
\begin{prop}[Theorem~1 of~\cite{grigsby_annular_2017}]\label{cobordism} Given a cobordism $S$ with $n$ odd index critical points and no even index critical points between two braid closures $ \widehat{\beta}$ and $ \widehat{\beta'}$ then the cobordism gives a bound on the difference of their $d_t$ invariants $\vert d_t( \widehat{\beta}) - d_t( \widehat{\beta'}) \vert \leq n$.\end{prop}
\begin{prop}[Proposition~4 of~\cite{grigsby_annular_2017}]\label{DisjointUnion}
If $\beta_0$ and $\beta_1$ are braids and $\beta_2$ is braid that is the disjoint union of $\beta_0$ and $\beta_1$ then $d_t(\widehat{\beta_2}) = d_t(\widehat{\beta_1}) + d_t(\widehat{\beta_2})$.
\end{prop}

Given a braid closure $\widehat{\beta}$, Plamenevskaya showed how to associate to $\widehat{\beta}$ a homology class $\psi(\widehat{\beta})$ in Khovanov homology which is well-defined up to sign~\cite{plamenevskaya_transverse_2006}. The homology class $\psi(\widehat{\beta})$ is invariant not only under braid conjugation but also under positive stabilization and destabilization and so it is an invariant of the transverse isotopy class of $\widehat{\beta}$. We recall the definition and a few properties of this invariant. The invariant is defined to be the homology class of the distinguished generator $v_-$ where all the circles of the oriented resolution are labeled with a $-$ sign in the Khovanov homology of $ \widehat{\beta}$.
\begin{prop}[Proposition 3.1 of~\cite{baldwin_categorified_2015}]\label{RV} If $\psi( \widehat{\beta}) \not = 0$ then $\beta$ is right-veering. \end{prop}
Note that there are examples of right-veering braids whose transverse invariants $\psi$ vanish~\cite{plamenevskaya_braid_2018}.
\begin{prop}[\cite{plamenevskaya_transverse_2006}]\label{slands} For any $n$-braid, if $s( \widehat{\beta}) - 1 = w( \widehat{\beta}) - n$ then $\psi( \widehat{\beta}) \not = 0$.\end{prop} This proposition includes all quasipositive braids as examples. More generally for any $n$-braid, if $d_t( \widehat{\beta})$ has slope $n$ at any point in the interval $[0,1)$ then $\psi( \widehat{\beta}) \not = 0$~\cite{grigsby_annular_2017}.

The transverse invariant $\psi$ is also functorial in the the following sense.
\begin{prop}[Theorem~4 of~\cite{plamenevskaya_transverse_2006}]\label{functorial} If $\beta'$ is a braid obtained from $\beta$ by adding a positive crossing then the crossing change induces a map $f: \Kh( \widehat{\beta}') \to \Kh( \widehat{\beta})$ and $f(\psi( \widehat{\beta}')) = \pm \psi( \widehat{\beta})$\end{prop} In particular if $\psi( \widehat{\beta}) \not = 0$ then $\psi( \widehat{\beta}') \not = 0$; equivalently if $\psi( \widehat{\beta}') = 0$ then $\psi( \widehat{\beta}) = 0$. 
 
For 3-braids, it is known exactly when the invariant $\psi$ is non-zero. This result was well known to experts but the author is not aware of a complete proof in the literature. The case when $d>1$ is also treated in~\cite{hubbard_note_2018}.  

\begin{lem}\label{psi3}
A $3$-braid $\beta$ has $\psi( \widehat{\beta}) \not = 0$ if and only if $\beta$ is conjugate to one of the following braids. 
\begin{enumerate}
\item $h^d\sigma_1 \sigma_2^{-a_1} \cdots \sigma_1 \sigma_2^{-a_n}$ with $a_i \geq 0$ and some $a_i > 0$ and $d > 0$.
\item $h^d \sigma_2^m$ with $m \in \Z$ and either $d = 0$ $m \geq 0$, $d = 1$  $m \geq -4$, or $d > 1$.
\item $h^d \sigma_1^m \sigma_2^{-1}$ with $m \in \{ -1 ,-2, -3\}$ and $d > 0$.
\end{enumerate}
\end{lem}
\begin{proof}

For the first family, when $d <1$ the braids are not right-veering and so their $\psi$ invariants vanish by Proposition~\ref{RV}. When $d = 1$ the argument that these braids are quasi-alternating gives you that the $\psi$ invariants are non-vanishing (see Remark 7.6 of~\cite{baldwin_khovanov_2010}). For $d>1$ these braids can be obtained from braids in this family with $d = 1$ by adding positive crossings and so their $\psi$ invariants are also non-vanishing by Proposition~\ref{functorial}. 

For the second family, when $d < 0 $ or $d = 0 $ and $m < 0$ the braids are all not right-veering so their $\psi$ invariants vanish by Proposition~\ref{RV}~\cite{plamenevskaya_braid_2018}. When $d = 1$ and $m = -5$ a calculation shows that the $\psi$ invariant vanishes and so it also vanishes when $m < -5$ by Proposition~\ref{functorial}~\cite{baldwin_khovanov_2010}. When $d = 1$ and $m = -4$, this braid is quasipositive so the $\psi$ invariant is non-zero by Proposition~\ref{slands} and then so are the $\psi$ invariants of all braids with $d = 1$ and $m > -4$ by Proposition~\ref{functorial}. Finally, if $d > 1$ then the braids can be obtained by adding positive crossings to braids in the first family with non-vanishing $\psi$ invariants and so the $\psi$ invariant is non-zero for these braids as well by Proposition~\ref{functorial}. 

For the third family, note that when $d \leq 0$ the braids are not right-veering and so their $\psi$ invariants vanish by Proposition~\ref{RV} and when $d>0$ the braids are quasipositive so their $\psi$ invariants are all non-zero by Proposition~\ref{slands}~\cite{plamenevskaya_braid_2018}.

\end{proof}

\section{Additional properties of the annular Rasmussen invariant}\label{additionalproperties}

In this section we aim to place constraints on the $d_t$ invariant of a braid. These constraints will limit the possible shapes of $d_t$. For a 3-braid, these constraints will imply that $d_t$ is determined entirely by the $s$ invariant and the writhe of the braid. 

\begin{prop}
If $\beta$ is a negative $n$ braid, then the slope $m_t( \widehat{\beta})$ is in the set $\{n , n -2 , \ldots , -n + 2 \}$ for $t \in (0,1)$.
\end{prop}
In~\cite{grigsby_annular_2017} it was proved that $m_t( \widehat{\beta})$ is in the set $\{n , n -2 , \ldots , -n + 2 , -n \}$ so we rule out the possibility of $m_t( \widehat{\beta}) = -n$ for $t \in (0,1)$.
\begin{proof}
For a negative braid closure $ \widehat{\beta}$, the oriented resolution is at the far right of the cube of resolutions, so the distinguished generators from the oriented resolution can only be homologous to linear combinations of other generators from the oriented resolution. The only generator $x$ with $k$-grading $n$ is the labeling of every circle with a $+$. The generator $x$ is contained in exactly one of the subcomplexes $L_1, L_2$ of the Lee chain complex, assume that it is in $L_1$. The homology class $[\fraks_0]$ is non-zero in both $L_1$ and $L_2$ by Proposition~\ref{subcomplex}, and so any representative of the homology class $[\fraks_0]$ has elements in both $L_1$ and $L_2$. Every distinguished generator $y$ in $L_2$ has $j_t$ grading less than $x$ for any $t \in [0,1]$ which means that $x$ is never the generator that determines $d_t( \widehat{\beta})$ for $t \in [0,1)$. The slope at $t$ is the negative $k$ grading of the generator determining $d_t( \widehat{\beta})$ so the slope $m_t( \widehat{\beta})$ is never $ -n$ for $t \in [0,1]$.

\end{proof}

The condition that negative $n$ braids don't have slope $-n$ is enough to show that all $n$ braids can't have slope $n$.

\begin{prop}
If $\beta$ is an $n$-braid, then the slope $m_t( \widehat{\beta})$ is in the set $\{n , n -2 , \ldots , -n + 2 \}$ for $ t \in (0,1)$.
\end{prop}
In~\cite{grigsby_annular_2017} it was proved that $m_t( \widehat{\beta})$ is in the set $\{n , n -2 , \ldots , -n + 2 , -n \}$ so we rule out the possibility of $m_t( \widehat{\beta}) = -n$ for $t \in (0,1)$.

\begin{proof}
The only generator with $k$ grading $n$ is the labeling of all circles in the oriented resolution with a $+$ and this generator lies on the line $w( \widehat{\beta}) - n(t-1)$ so it is enough to show that $d_t( \widehat{\beta})$ does not lie on this line for $t \in [0,1)$.

Let $\beta$ have $n_+$ positive crossings and $n_-$ negative crossings and $\beta'$ be the negative $n$ braid that is the result of removing all positive crossings from $\beta$. There is a cobordism from $\beta$ to $\beta'$ with $n_+$ odd-index annular critical points, one for each positive crossing removed from $\beta$. By Proposition~\ref{cobordism}, the cobordism gives the bound  $\vert d_t( \widehat{\beta}) - d_t( \widehat{\beta}') \vert \leq n_+$. For any $t \in [0,1)$ the previous proposition implies that $d_t(\beta') < -n_-  -n(t-1)$ so then $d_t( \widehat{\beta}) < n_+ -n_-  - n(t-1) = w - n(t-1)$ and $d_t( \widehat{\beta})$ doesn't lie on this line for $t \in [0,1)$.

\end{proof}

The following technical proposition will be used to provide strong restrictions on the shape of the $d_t$ invariant.

\begin{prop}\label{techprop}
Let $x,y$ are homogenous elements in different subcomplexes $L_1$ and $L_2$. Suppose in some neighborhood $(t_0, t_2)$ that $j_t(y) < j_t(x)$ on $(t_0 , t_1)$ and $j_t(x) < j_t(y)$ on $(t_1 , t_2)$. Additionally suppose that if $a \in L_1$ with $x$ and $j_t(a) > j_t(x)$ on $(t_1, t_2)$ then $j_t(a) > j_t(x)$ on $(t_0 , t_2)$ and similarly if $b \in L_2$ with $y$ and $j_t(b) > j_t(y)$ on $(t_0 , t_1)$ then $j_t(b) > j_t(y)$ on $(t_1, t_2)$. Then it is not possible to have $x$ determine $d_t$ on $(t_0 , t_1)$ and $y$ determine $d_t$ on $(t_1, t_2)$.
\end{prop}

Note that a sufficient condition on when the assumptions of the theorem are met is when no other generator's line passes through the point of intersection of the lines of $x$ and $y$.

\begin{proof}
Suppose that $x$ determines $d_t$ on $(t_0 , t_1)$ and $y$ determines $d_t$ on $(t_1, t_2)$. Then there is a cycle $c_1$ representing $[\fraks_0]$ with $c_1 = x + \sum a_i + \sum b_j$ with $a_i \in L_1$ and $b_j \in L_2$ and $j_t(x) < j_t(b_j) $ for $t \in (t_0 , t_1)$. Similarly there is a cycle $c_2$ representing $[\fraks_0]$ with $c_2 = y + \sum a'_i + \sum b'_j$ with $a'_i \in L_1$ and $b'_j \in L_2$ and $j_t(y) < j_t(a'_i) $ for $t \in (t_1 , t_2)$. Then $j_t(x)<j_t(a'_i)$ for $t \in (t_0 , t_2)$ and $j_t(y) < j_t(b_j)$ for $t \in (t_0 , t_2)$. Then $c_3 = \sum a'_i + \sum b_j$ is a representative of $[\fraks_0]$ with $j_t(x) < j_t(c_3) $ for $t \in (t_0 , t_1)$ and $j_t(y) < j_t (c_3)$ for $t \in (t_1 , t_2)$. Constructing this representative of $[\fraks_0]$ with higher $j_t$ grading contradicts that $x$ and $y$ determined $d_t$ on these intervals.
\end{proof}

Informally, Proposition~\ref{techprop} means that you can't see a change from being determined by an element in one subcomplex to being determined by an element in the other subcomplex with lower $k$ grading/higher $j_t$ slope under certain additional assumptions.

\begin{rem}
It was pointed out to the author by Stephan Wehrli that it should also be possible to define two ``partial" $d_t$ invariants by using the portions of $[\mathfrak{s}_o]$ in each of the subcomplexes $L_1$ and $L_2$. In this case the $d_t$ invariant would be the minimum of the two ``partial" invariants for each $t$. This definition would provide a slightly different perspective to view the restrictions on the $d_t$ invariant provided by Proposition~\ref{techprop}.
\end{rem}

The $d_t$ invariant of a braid is determined by the $j_t$ grading of some generator of the braid's Lee chain complex. Noticing this and a few properties of the $d_t$ invariant, it is possible to restrict the shape of the $d_t$ invariant of a braid.

\begin{thm}\label{dtn}
For an $n$-braid $\beta$, there are only finitely many possible shapes that $d_t( \widehat{\beta})$ can take and there is a method for enumerating all the possibilities.
\end{thm}
\begin{proof}
 
At $t=1$ the value of $d_t( \widehat{\beta})$ is known, $d_1( \widehat{\beta}) = w( \widehat{\beta})$ and in the interval $[0,1)$ the slope $m_t( \widehat{\beta})$ is bounded between $n$ and $-n + 2$. So in this interval, $d_t( \widehat{\beta})$ lies in the triangular region bounded by the lines passing through $(1, w( \widehat{\beta}))$ with slopes $n$ and $-n + 2$. A start to understanding the possible shapes of $d_t$ is listing all possible $j_t$ gradings of generators in this region for some $t \in (0,1)$.

There are generators with $j_t$ gradings on the line of slope $-n +2$ from $(0, w( \widehat{\beta}) + n -2)$ to $(1, w( \widehat{\beta}))$ and a single generator with $j_t$ grading on the line from $(0, w( \widehat{\beta}) - n )$ to $(1, w(\beta))$. Additionally there could be generators with $j_0$ grading from the set $\{w(\beta) - n + 2 , w(\beta) - n + 4, \ldots , w( \widehat{\beta}) + n - 4\}$ with any slope from the set $\{ n - 2, n - 4, \ldots , -n +2\}$. Now every possible shape of $d_t(\beta)$ must be some path along the lines coming from these generators which stays in the triangular region, never has $t$ decrease, and ends at the point $(1,w( \widehat{\beta}))$.

Using the restriction from Proposition~\ref{techprop}, some of these paths can be ruled out as possibilities for the shape of $d_t( \widehat{\beta})$.

\end{proof}

\begin{example}
\begin{figure}[h!]

  \centering
    \includegraphics[width=0.4\textwidth]{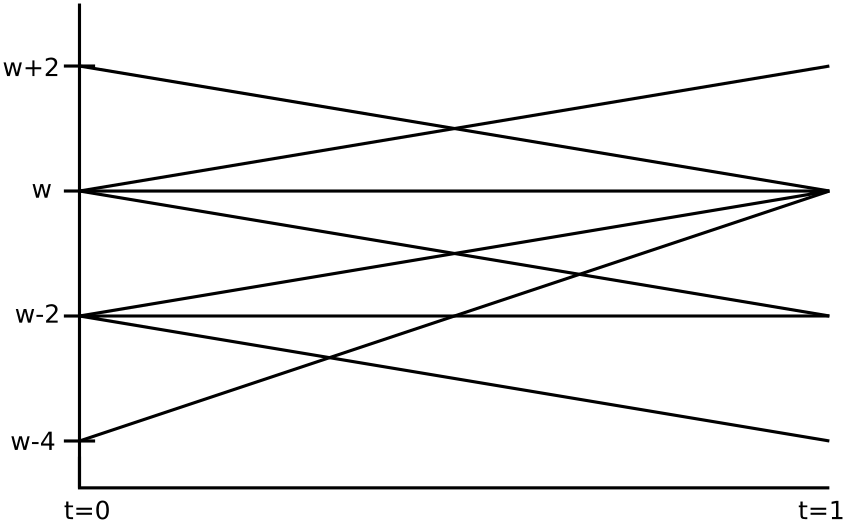}
      \caption{Possible generators for a 4-braid}\label{4braidexa}
\end{figure}

As an example of Theorem~\ref{dtn}, we can enumerate all possible shapes of $d_t( \widehat{\beta})$ when $\beta$ is a 4-braid. For a 4-braid, the shape of $d_t$ is restricted to paths to $(1,w( \widehat{\beta}))$ in Figure~\ref{4braidexa}. Examining the figure, there are ten possible paths; the six with non-constant slope are shown in Table~\ref{fig:4braids}. Some of these paths can not be the shape of $d_t$ however because of the restrictions on $d_t$ from Proposition~\ref{techprop}. Specifically the first and second paths in the top row and the second path in the bottom row can not be the shape of $d_t$ because of Proposition~\ref{techprop}.

\begin{table}
  \centering
  \begin{tabular}{ccc}
 \includegraphics[width=0.3\textwidth]{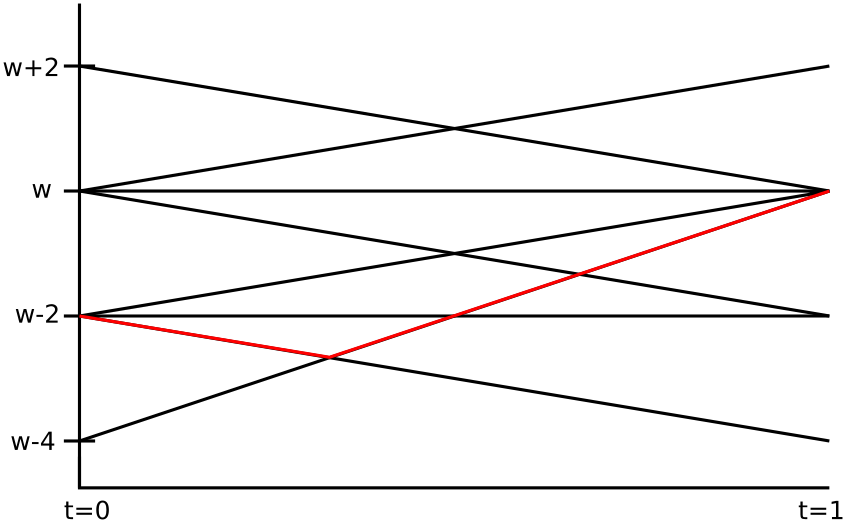}&\includegraphics[width=0.3\textwidth]{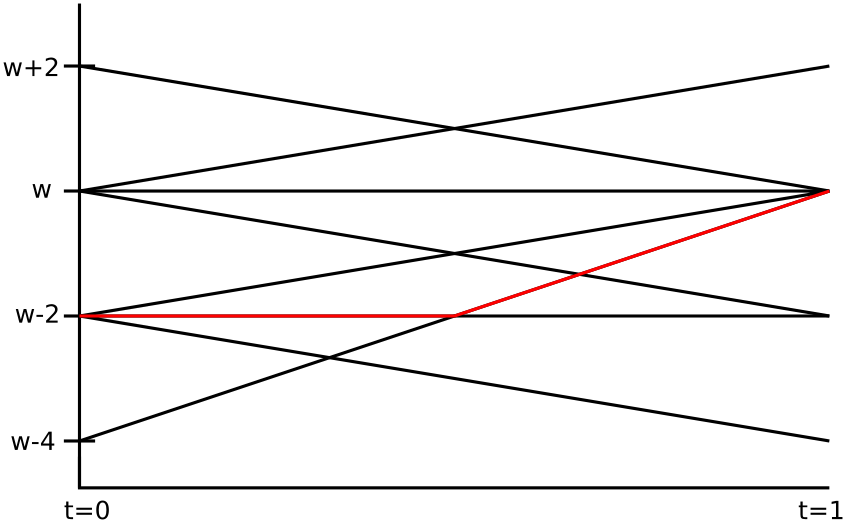}&   \includegraphics[width=0.3\textwidth]{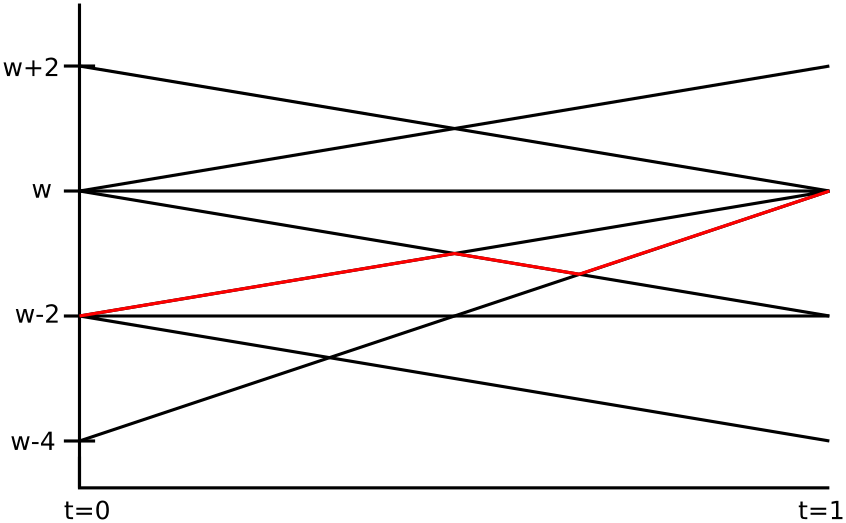}\\
 \\
 \includegraphics[width=0.3\textwidth]{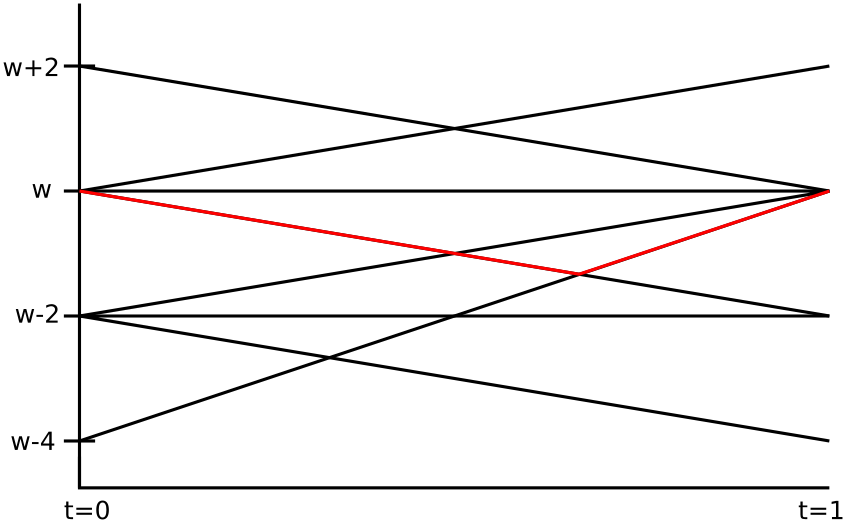}&\includegraphics[width=0.3\textwidth]{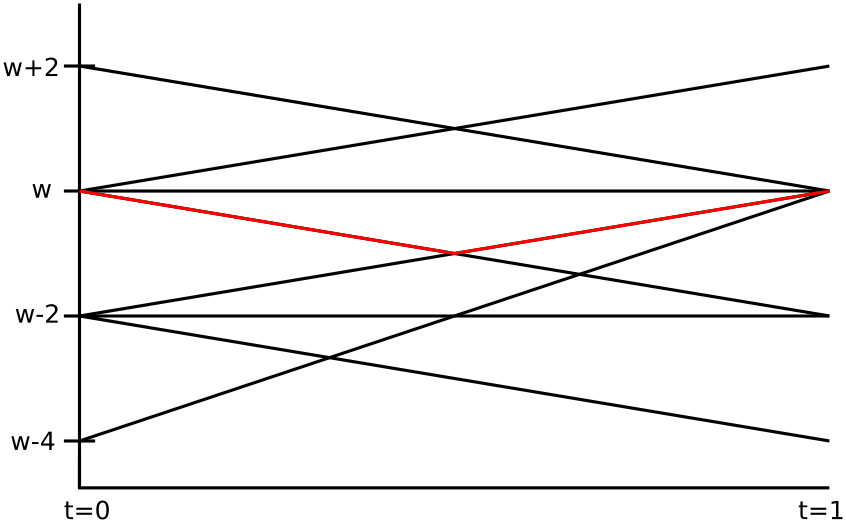}&   \includegraphics[width=0.3\textwidth]{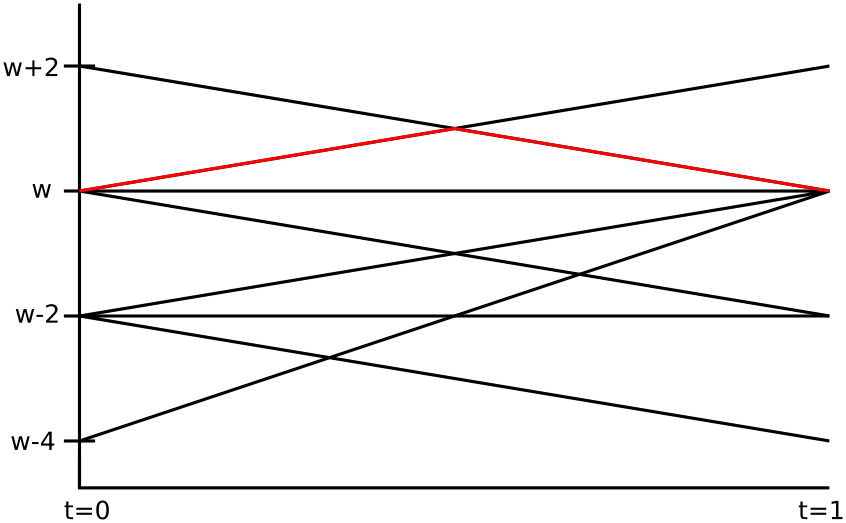}\\

   \end{tabular}
 \caption{The six possible shapes of $d_t$ of a 4-braid with non-constant slope}\label{fig:4braids}
\end{table}

\end{example}

Replicating the same process for 3-braids provides even stronger restrictions on the shape of the $d_t$ invariant.

\begin{thm}\label{dt3}
When $\beta$ is a 3-braid, $d_t( \widehat{\beta})$ is entirely determined by if $s( \widehat{\beta})= w(\widehat{\beta})-2$, $s( \widehat{\beta})= w(\widehat{\beta})$, or $s( \widehat{\beta})= w(\widehat{\beta})+2$.
\end{thm}

\begin{proof}
\begin{figure}

  \centering
    \includegraphics[width=0.4\textwidth]{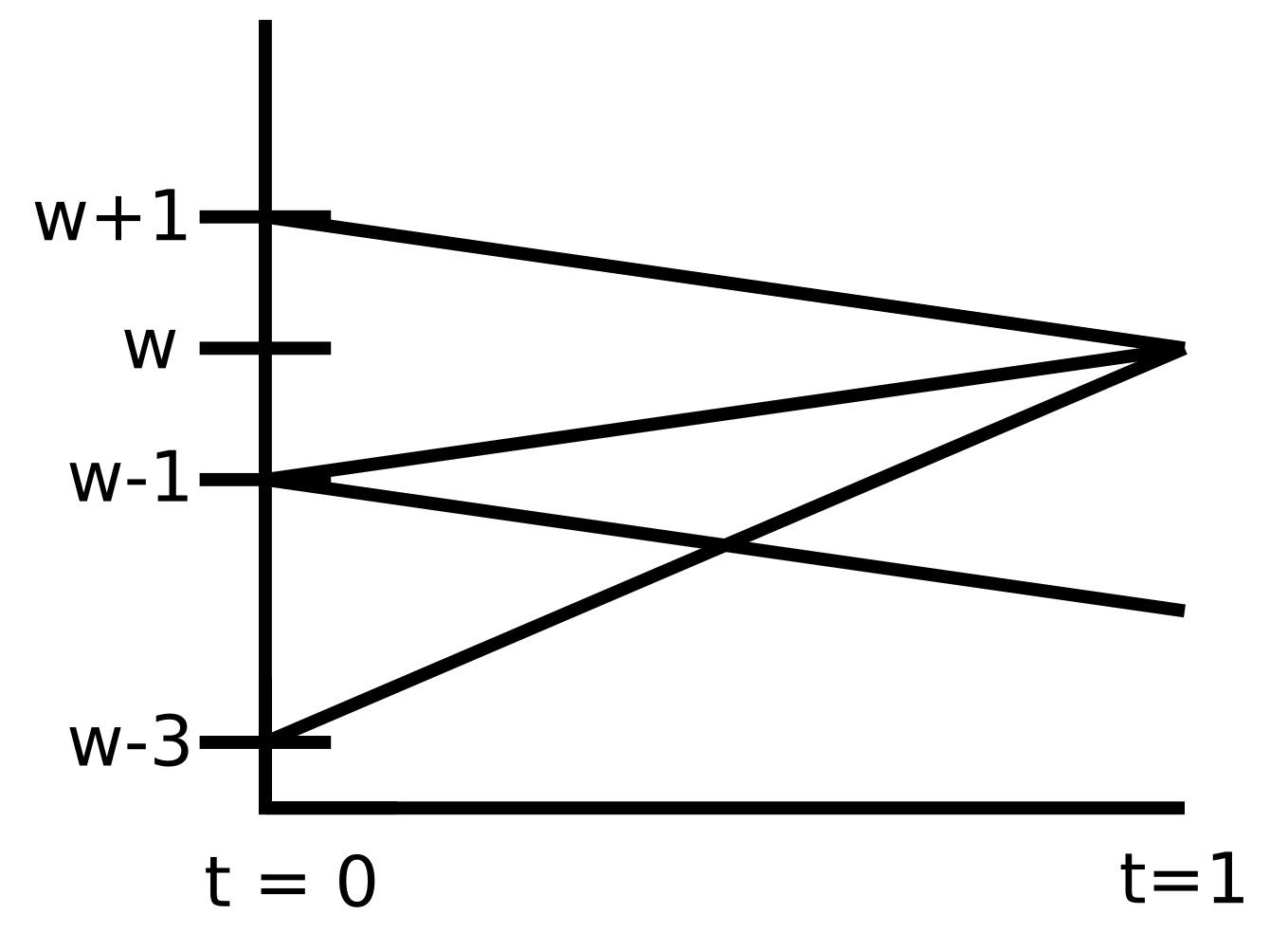}
      \caption{Possible generators for a 3-braid}\label{3braidexa}
\end{figure}

Following the process described in Theorem~\ref{dtn} for a 3-braid shows that the shape $d_t$ is restricted by the paths to $(1,w( \widehat{\beta}))$ in Figure~\ref{3braidexa}. Examining the figure, there are exactly four paths, three with constant slope and a single path that starts at $w-1$ with slope $-1$ until it intersects the line with slope $3$ and then follows the line with slope $3$ to the endpoint $(1,w( \widehat{\beta}))$. Note that the generators on these two lines are in different subcomplexes. So then this single path with non-constant slope can not be the shape of $d_t$ because it is ruled out by Proposition~\ref{techprop}. 

\end{proof}

\subsection{Application to the Upsilon invariant}\label{Upsilon}

The concordance invariant $\Upsilon_K(t)$ of a knot is defined in an algebraically analogous way to the annular Rasmussen $d_t$ invariant but with a $\mathbb{Z} \oplus \mathbb{Z}$ filtered complex $\CFK^\infty(K)$ coming from knot Floer homology. Because the two invariants are defined in a similar way, some of the perspective used to prove Theorem~\ref{dtn} can be used to prove a similar statement about $\Upsilon_K(t)$.

The invariant $\Upsilon_K(t)$ was originally defined by Ozsv\'{a}th, Stipsicz, and Szab\'{o} in~\cite{ozsvath_concordance_2017} and reinterpreted by Livingston in~\cite{livingston_notes_2017}. A complete definition of the invariant $\Upsilon_K(t)$ would be outside the scope of the present paper but a reader can refer to Livingston's article~\cite{livingston_notes_2017} for a definition in terms of a $\mathbb{Z} \oplus \mathbb{Z}$ filtered complex $\CFK^\infty(K)$ filtered by the Alexander filtration $j$ and the Algebraic filtration $i$. The following facts about $\Upsilon_K(t)$ are all that are needed for our application.

\begin{prop}[\cite{livingston_notes_2017}]\label{UpsilonFacts}
The piecewise linear function $\Upsilon_K(t)$ has the following properties:
\begin{enumerate}
\item $\Upsilon_K(0) = 0$.
\item  $\Upsilon'_K(t) = \pm ( i(x) - j(x)) $ for some generator $x \in \CFK^\infty(K)$.
\item $\Upsilon'_K(t) \in \{ -g(K) , -g(K) + 1 , \ldots, g(K) - 1, g(K) \}$ where $g(K)$ is the 3-dimensional genus of $K$.
\item  $\Upsilon_K(t)$ is symmetric about $t = 1$.
\item  $\Upsilon_K(t)$ is a concordance invariant.
\end{enumerate}
\end{prop} 

Because of the symmetry about $t=1$ we restrict our focus to the value of $\Upsilon_K(t)$ for $t$ in the interval $[0,1]$.

\begin{thm}\label{UpsilonFinite}
For a fixed concordance genus $c$ there are finitely many possibilities for $\Upsilon_K(t)$ for any $K$ of concordance genus $c$ and a method for enumerating all the possibilities.
\end{thm}

\begin{proof}

Because $\Upsilon_K(t)$ is a concordance invariant, we can assume, after possibly choosing a different knot in its concordance class, that the 3-dimensional genus of $K$ is $c$.

At $t = 0$ the filtration is the Algebraic filtration $i$ which is a $\mathbb{Z}$ filtration so all the generators sit at some integer filtration level. Additionally, the slope of the generators is in the set $\{ - c , -c + 1, \ldots , c - 1 , c\}$ and so then also at $t = 1$ the generators all sit at some integer filtration level.

We know that $\Upsilon_K(0) = 0$ and so then $-tc \leq \Upsilon_K(t) \leq tc$ for $t \in [0,1]$ by the bounds on the slope $\Upsilon'_K(t)$. So then we just draw all the possible generators that will sit in this triangle between the points $(0,0)$, $(1 , c)$ and $(1 , -c)$. There are two lines from $(0,0)$ to $(1,c)$ and $(1, -c)$. For all the other generators at $t = 1$ the lines must pass through a point in the set $\{ (1, -c +1) , (1 , -c+2) , \ldots , ( 1 , c - 2) , (1,c-1)\}$ and have slope in the set $\{ - c , -c + 1, \ldots , c - 1 , c\}$. So if we draw all these lines then the possibilities for $\Upsilon_K(t)$ are all paths with non-decreasing values of $t$ from $(0,0)$ to the line $t = 1$ along these lines.

\end{proof}
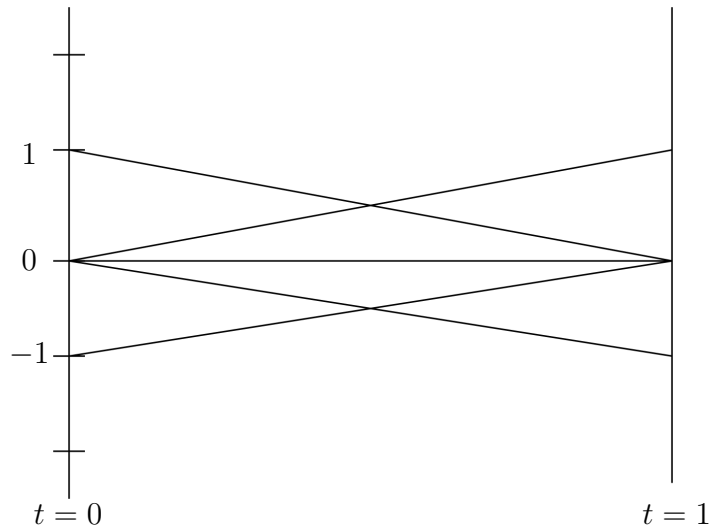
\begin{figure}

  \centering
     \begin{overpic}[width=0.5\textwidth]{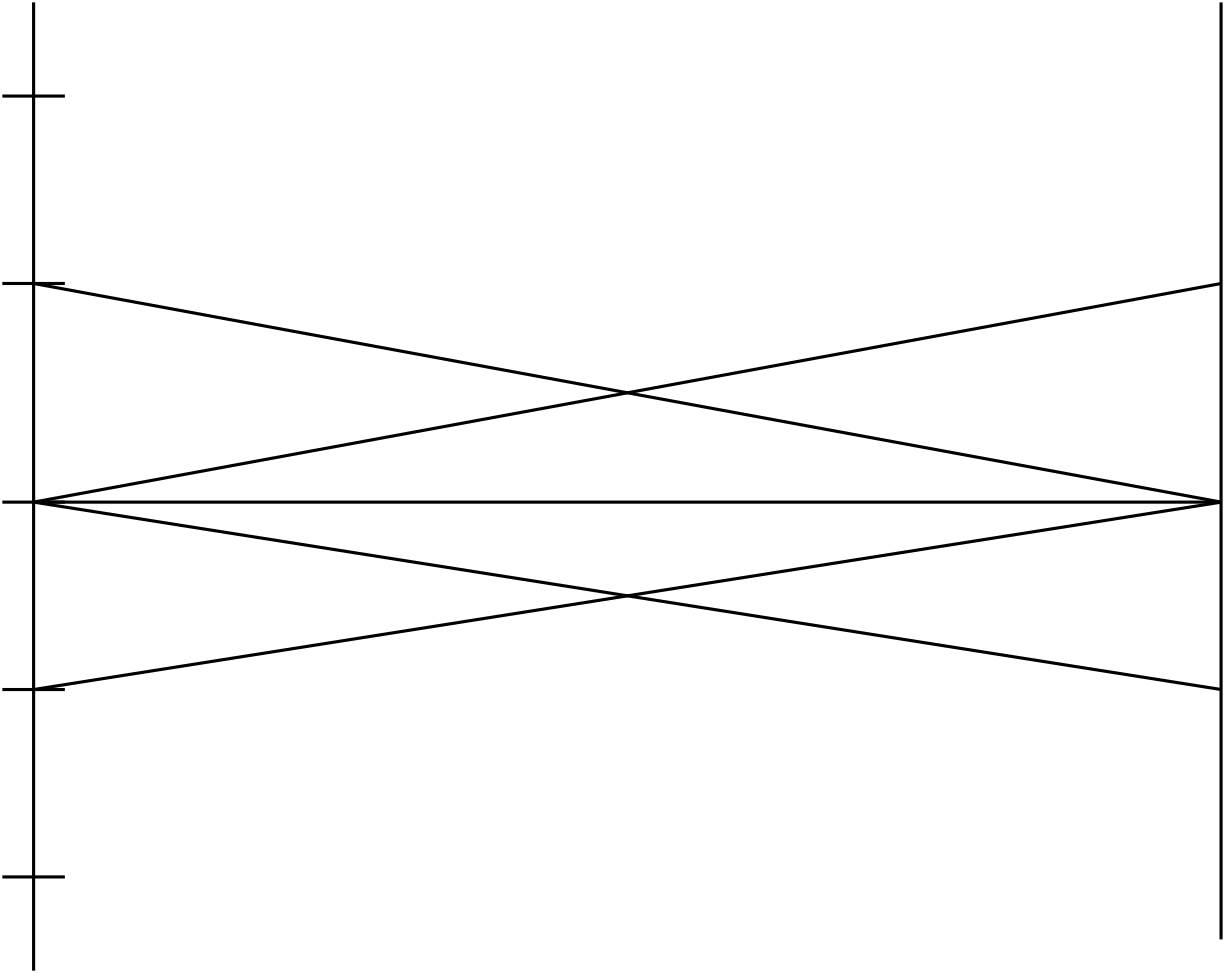}
    \put(-5,37){$0$}
     \put(-5,54){$1$}
     \put(-7,22){$-1$}
     \put(-3,-4){$t = 0$}
      \put(95,-4){$t = 1$}
    \end{overpic}
      \caption{Possible shapes of $\Upsilon_K(t)$ for $g_c(K) = 1$}\label{Upsilon1}
\end{figure}
\begin{example}
When $c = 1$ applying the process of drawing lines from Theorem~\ref{UpsilonFinite} gives the graph shown in Figure~\ref{Upsilon1}. Counting all paths shows that there are exactly five possibilities for $\Upsilon_K(t)$ if $K$ has concordance genus 1.
\end{example}

\begin{rem}
The number of possibilities for $\Upsilon_K(t)$ given by the process in Theorem~\ref{UpsilonFinite} seems to grow much faster than the number of possible shapes of $d_t(\widehat{\beta})$. For example focusing on $K$ with concordance genus 2, the restrictions from Theorem~\ref{UpsilonFinite} still allow for over 50 possibilities for $\Upsilon_K(t)$. 
\end{rem}

\section{The annular Rasmussen invariants of 3-braids}\label{s3braids}

For a 3-braid $\beta$, the $s$ invariant of $ \widehat{\beta}$ is either $w( \widehat{\beta}) - 2$, $w( \widehat{\beta})$, or $w( \widehat{\beta}) + 2$ depending on if the slope of $d_t$ in $(0,1)$ is $3,1$, or $-1$. The following theorems completely classify when each of these possibilities occur.

\begin{thm}\label{slope3}
A 3-braid $\beta$ has $s( \widehat{\beta}) = w( \widehat{\beta}) - 2 $ if and only if $\beta$ is conjugate to a braid of the form: \begin{enumerate}
\item $h^d\sigma_1 \sigma_2^{-a_1} \cdots \sigma_1 \sigma_2^{-a_n}$ with $a_i \geq 0$ and some $a_i > 0$ and $d > 0$.
\item $h^d \sigma_2^m$ with $m \in \Z$ and either $d = 0$ $m \geq 0$, $d = 1$  $m \geq -4$, or $d > 1$.
\item $h^d \sigma_1^m \sigma_2^{-1}$ with $m \in \{ -1 ,-2, -3\}$ and $d > 0$.
\end{enumerate}
\end{thm}

\begin{thm}\label{slope-1}
A 3-braid $\beta$ has $s( \widehat{\beta}) = w( \widehat{\beta}) + 2 $ if and only if the mirror of $\beta$ is conjugate to a braid of the form:
\begin{enumerate}
\item $h^d\sigma_1 \sigma_2^{-a_1} \cdots \sigma_1 \sigma_2^{-a_n}$ with $a_i \geq 0$ and some $a_i > 0$ and $d > 0$.
\item $h^d \sigma_2^m$ with $m \in \Z$ and either $d = 1$  $m \geq -3$, or $d > 1$.
\item $h^d \sigma_1^m \sigma_2^{-1}$ with $m \in \{ -1 ,-2, -3\}$ and $d > 0$.
\end{enumerate}

\end{thm}

\begin{thm}\label{slope1}
All other 3-braids have $s( \widehat{\beta}) = w( \widehat{\beta}) $. 
\end{thm}

\begin{rem}
The example of the braids $h^{1}\sigma_2^{-4}$ and $h^{-1}\sigma_2^4$ show that there is no simple formula for the behavior of $d_t$ under mirroring for links. It is not known if there is a simple formula for the behavior of $d_t$ under mirroring for knots.
\end{rem}

The following property of the $d_t$ invariant of 3-braids is useful for the computation $s$ invariant of 3-braid closures.

\begin{lem}\label{crossings}
Adding positive crossings to a 3-braid can only preserve or increase the slope of $d_t( \widehat{\beta})$ in the interval $(0,1)$. Similarly adding negative crossings to a 3-braid can only preserve or decrease the slope of $d_t( \widehat{\beta})$ in the interval $(0,1)$.
\end{lem}

\begin{rem}
This fact about adding crossings and how that changes the slope of the $d_t$ invariant is a specific property of 3-braids and not true in general. For example, for 4-braids it should be possible to begin with a braid whose shape is shown in the top right of Table~\ref{fig:4braids} add a positive crossing and get a braid whose shape has constant slope of $2$. This would decrease the slope for some values of $t$.
\end{rem}
\begin{proof}
Let $\beta$ be a 3-braid and consider $\beta\sigma_i$ the result of adding a positive crossing, notice that there is a cobordism from $\beta$ to $\beta\sigma_i$ with a single index one critical point. Now by Proposition~\ref{cobordism}, $d_1( \widehat{\beta}) +1= w( \widehat{\beta}) +1  = d_1( \widehat{\beta \sigma_i})$ which means that if $d_0(\widehat{\beta\sigma_i}) = d_0(\widehat{\beta}) +1$ then the two braids have the same $d_t$ slope. The bound on $d_t$ from cobordisms gives that $d_0(\widehat{\beta\sigma_i}) \leq d_0( \widehat{\beta}) +1$. If equality holds then $d_t$ of $\widehat{\beta\sigma_i}$ has the same slope as that of $ \widehat{\beta}$ and if equality does not hold then $d_t$ of $\widehat{\beta\sigma_i}$ has a more positive slope than that of $ \widehat{\beta}$. The same argument proves the case of adding a negative crossing. 
\end{proof}

With this lemma, we are now ready to prove Theorems~\ref{slope3},~\ref{slope-1},~and~\ref{slope1}.

\begin{proof}[Proof of Theorem~\ref{slope3}]
If a 3-braid $\beta$ has $s( \widehat{\beta}) = w( \widehat{\beta}) - 2$ then $\psi( \widehat{\beta}) \not= 0$ by Proposition~\ref{slands} so the braids listed are the only possible 3-braids with $s( \widehat{\beta}) = w( \widehat{\beta}) - 2$ by Lemma~\ref{psi3}.

The braids in the family 1 with $d = 1$ are all quasi-alternating and have $\psi( \widehat{\beta}) \not = 0$ so they have $s( \widehat{\beta}) = w( \widehat{\beta}) - 2 $. For $d > 1$ then the braids can be obtained from a braid with $d = 1$ by adding positive crossings and so they also have $s( \widehat{\beta}) = w( \widehat{\beta}) - 2 $.

In family 2, when $d > 1$ then the braids can be obtained from the braids in the first family by adding positive crossings and so these braids have maximal $d_t$ slope on (0,1) and $s( \widehat{\beta}) = w( \widehat{\beta}) - 2$ by Lemma~\ref{crossings}. If $d=1$ and $m = -4$ then the braid is quasipositive and so $s( \widehat{\beta}) = w( \widehat{\beta}) - 2$ by Proposition~\ref{slands}. Adding positive crossings shows that the $s( \widehat{\beta}) = w( \widehat{\beta}) - 2$ also holds when $m > -4$ by Lemma~\ref{crossings}. Finally if $d = 0$ and $m\geq 0$ then the braids are quasipositive and so have $s( \widehat{\beta}) = w( \widehat{\beta}) -2$.

The braids in family 3 are all quasipositive and so they have $s( \widehat{\beta}) = w( \widehat{\beta}) - 2$ by Proposition~\ref{slands}.
\end{proof}

\begin{proof}[Proof of Theorem~\ref{slope-1}]
First, notice that if $\beta$ is a 3-braid whose closure is a knot and $m(\beta)$ has $\psi(m( \widehat{\beta})) \not = 0$ then $s( \widehat{\beta}) = - s(m( \widehat{\beta})) = -(w(m( \widehat{\beta})) - 2) = w( \widehat{\beta}) +2$ by Proposition~\ref{mirror}. More generally if $\beta'$ is any 3-braid where it is possible to obtain $\beta$ as above by adding positive generators to $\beta'$ then $d_0( \widehat{\beta}') = w + 1$ because of Lemma~\ref{crossings}. Also, adding positive generators to $\beta'$ is equivalent to adding negative generators to $m(\beta')$ and arriving at $m(\beta)$.

If $m(\beta')$ is conjugate to a braid in family 1 and has $d>0 $ for $m(\beta')$ then you can add negative crossings to $m(\beta')$ to get to the braid $m(\beta) = h^d \sigma_1 \sigma_2^{-n}$ with $d>1$ and then after maybe adding an additional negative crossing $m(\beta)$ closes up to a knot and $m( \widehat{\beta})$ also has $\psi$ non-vanishing. Then $\beta$ has $s( \widehat{\beta}) = w( \widehat{\beta}) + 2$ and so then the same holds for $ \widehat{\beta}'$.

If $m(\beta')$ is conjugate to a braid in family 2 with $\psi(m( \widehat{\beta}')) \not= 0 $ and $d > 1$ then you can add negative crossings to the mirror to get to a braid $m(\beta)$ in family 1 with $d \geq 1$ by canceling out all of a full twist except a single $\sigma_1$. Then $\beta$ has $s( \widehat{\beta}) = w( \widehat{\beta}) + 2$ and so then the same holds for $ \widehat{\beta}'$.

For $\beta$ with $d = 1$ and $m = -3$ then the Khovanov homology of $ \widehat{\beta}$ in homological grading 0 has dimension 2 and is supported in quantum gradings -2 and 0. This implies that $s( \widehat{\beta}) = -1 = w( \widehat{\beta}) + 2$. For $\beta'$ with $d = 1$ and $m > -3$, it is possible to add negative crossings to $m(\beta')$ to arrive at $m(\beta)$ and so $s( \widehat{\beta}') = w( \widehat{\beta}') + 2$ as well.

If $m(\beta')$ is conjugate to a braid in family 3 and has $\psi(m( \widehat{\beta}')) \not = 0 $ then you can add negative crossings to $m(\beta')$ to get to $m(\beta) = h^d \sigma_1^{-3} \sigma_2^{-1}$ which is a knot with $\psi( m(\widehat{\beta})) \not = 0$. So then $\beta$ has $s( \widehat{\beta}) = w( \widehat{\beta}) + 2$ and so then the same holds for $ \widehat{\beta}'$.

To complete the proof that these are the only 3-braids with $s( \widehat{\beta}) = w( \widehat{\beta}) + 2$ we will show that all other $3$ braids have $s( \widehat{\beta}) = w( \widehat{\beta})$.
\end{proof}

\begin{proof}[Proof of Theorem~\ref{slope1}] 

If $\beta$ is conjugate to a braid in family 2 with $d= 0$ and $m < 0$ then $\beta$ is conjugate to a split braid. Specifically it is the union of a trivial braid on a single strand $\mathbbm{1}_1$ and a braid $\alpha$ with $\vert m\vert $ half-twists on two strands. Proposition~\ref{DisjointUnion} implies that $d_0(\widehat{\beta}) =d_0(\widehat{\mathbbm{1}_1}) + d_0(\widehat{\alpha}) = -1 + ( w(\widehat{\alpha)}  ) = w(\widehat{\beta})-1$ and so then $s(\widehat{\beta}) = w(\widehat{\beta})$.

If $\beta$ is conjugate to the braid $h^{-1}\sigma_2^4$ in family 2 then an explicit computation shows that $s(\widehat{\beta}) = -2 = w(\widehat{\beta})$. This computation is included in Section~\ref{computation}.

All remaining 3-braids have $\psi(\widehat{\beta}) = 0 = \psi(m(\widehat{\beta}))$. 

First notice that if $\beta$ is a 3-braid whose closure is a knot and $\psi(\widehat{\beta})= 0  = \psi(m(\widehat{\beta}))$ then $s(\widehat{\beta}) = w(\widehat{\beta})$ by Proposition~\ref{mirror}. More generally if $\beta'$ is a 3-braid with $\psi(\widehat{\beta}') = 0 $ and it is possible to add negative crossings to $\beta'$ to arrive at $\beta$ with $s(\widehat{\beta}) = w(\widehat{\beta})$ then $s(\widehat{\beta}') = w(\widehat{\beta}')$ as well. This is because by Lemma~\ref{crossings} we know that $s(\widehat{\beta}') = w(\widehat{\beta}')$ or $w(\widehat{\beta}') - 2$ but the fact that $\psi(\widehat{\beta}')  = 0$ rules out the second possibility by Proposition~\ref{slands}. Finally, adding negative crossings to $\beta'$ is the same as adding positive crossings to $m(\beta')$.

Notice that if we add negative crossings to $\widehat{\beta}'$ to arrive at $\widehat{\beta}$ and $\psi(\widehat{\beta}') = 0 = \psi(m(\widehat{\beta}'))$ then the functoriality of $\psi$ implies that $\psi(\widehat{\beta})= \pm f(\psi(\widehat{\beta}')) = 0$ so we only need to check that $\psi(m(\widehat{\beta})) = 0$ as well if $\beta$ closes up to a knot. 

For $m(\beta')$ is conjugate to a braid in family 1 with $\psi(\widehat{\beta}') = 0 = \psi(m(\widehat{\beta}')) $, you can add positive crossings to $m(\beta')$ to get to $h^d \sigma_1^k \sigma_2^{-1} \sigma_1^j$ with $d \leq 0$ and after possibly adding another positive crossing this is a braid $m(\beta)$ in family 1 whose closure is a knot and has $\psi(\widehat{\beta}) = 0 = \psi(m(\widehat{\beta}))$. So then $s(\widehat{\beta}') = w(\widehat{\beta}')$. Note that we only require that $d \leq 0$ and don't identify specific values of $k,j,$ and $d$ which satisfy $\psi(\widehat{\beta}) = 0 = \psi(m(\widehat{\beta}))$ but as long as the braid $\widehat{\beta}'$ we started with satisfies $\psi(\widehat{\beta}') = 0 = \psi(m(\widehat{\beta}')) $ then we know that the braid $\widehat{\beta}$ that we obtain will as well.

For $m(\beta')$ is conjugate to a braid in family 2 with $\psi(\widehat{\beta}') = 0 = \psi(m(\widehat{\beta}'))$, we have that either $\beta' = h \sigma_2^{k}$ for $k \leq -5$ or $\beta' = h^{-1} \sigma_2^{k}$ for $k \geq 5$. If $\beta'$ is of the form $\beta' = h^{-1} \sigma_2^{k}$ for $k \geq 5$ then it is possible to add negative crossings to $\beta'$ and arrive at $\beta = h^{-1} \sigma_2^{5}$. Examining the Khovanov homology of $\widehat{\beta}$ in homological grading $0$, it has dimension $1$ in quantum grading $-2$ and dimension $2$ in quantum grading $0$ and is zero in all other quantum gradings. This means that $s(\widehat{\beta}) = -1 = w(\widehat{\beta})$ and so then $s(\widehat{\beta}') = w(\widehat{\beta}')$ as well for all $\beta' = h^{-1} \sigma_2^{k}$ with $k \geq 5$.

If $\beta' = h \sigma_2^{k}$ for $k \leq -5$ then you can rewrite $\beta'$ as the word $\sigma_1 \sigma_2^2 \sigma_1 \sigma_2^{k+2}$ and then you can add negative crossings to $\beta'$ to arrive at the braid $\beta = \sigma_2^{k+2} $ which has $s(\widehat{\beta}) = w(\widehat{\beta})$ and so then $s(\widehat{\beta}') = w(\widehat{\beta}')$.

For $m(\beta')$ is conjugate to a braid in family 3 with $\psi(\widehat{\beta}') = \psi(m(\widehat{\beta}')) = 0$, you can add positive crossings to $m(\beta')$ and get to the braid $m(\beta) = h^d \sigma_1^{-1} \sigma_2^{-1}$ with $ d \leq 0$ which is a knot with $\psi(m(\widehat{\beta})) = 0 =\psi(\widehat{\beta})$, so then $s(\widehat{\beta}') = w(\widehat{\beta}')$.

\end{proof}

\section{Comparisons with other invariants of 3-braids}\label{compare} 

The results of Theorems~\ref{psi3},~\ref{slope3} and~\ref{slope-1} show a close connection between $d_t$ and $\psi$ for 3-braids.

\begin{cor}
A 3-braid $\beta$ has $\psi(\widehat{\beta}) \not= 0 $ if and only if $d_t(\widehat{\beta})$ has constant maximal slope, i.e. $s(\widehat{\beta}) = w(\widehat{\beta}) - 2$.
\end{cor}

\begin{cor}
If $\beta$ is a non-split 3-braid with $\psi(m(\widehat{\beta})) \not = 0$ then $\beta$ is conjugate to $h^{-1} \sigma_2^4$ or $d_t(\widehat{\beta})$ has constant minimal slope, i.e. $s(\widehat{\beta}) = w+2$.
\end{cor}

An open problem is if $\psi$ is an ``effective" transverse invariant, that is if it contains more information about transverse links than the self-linking number and the smooth link type. A weaker question which is also unknown is if vanishing/non-vanishing of $\psi$ contains more information than the $s$ invariant and the self-linking number. Birman and Menasco showed that there are transversly non-isotopic 3-braid closures with the same underlying smooth link type and the same self linking number~\cite{birman_stabilization_2006} so it is meaningful to ask if a transverse invariant is effective for 3-braid closures.
\begin{cor}\label{NotEffective}
The invariant $\psi$ is not effective for 3-braid closures. In particular, the vanishing/non-vanishing of $\psi$ for 3-braids is determined by the $s$ invariant and the self-linking number.
\end{cor}

Along with the $d_t$ invariant and $\psi$ invariant, there are other invariants that have been previously computed for 3-braids. Two examples are the transverse invariant from knot Floer homology $\widehat{\Theta}$ and the contact invariant of double branched covers of 3-braids $c(T,\phi)$.

\begin{thm}[Theorem 4.1 of~\cite{plamenevskaya_braid_2018} ]
For a 3-braid $\beta$, the invariant $\widehat{\Theta}(\widehat{\beta}) $ is non-zero if and only if $\beta$ is conjugate to a braid of the following form:
\begin{enumerate}
\item $h^d\sigma_1 \sigma_2^{-a_1} \cdots \sigma_1 \sigma_2^{-a_n}$ with $a_i \geq 0$ and some $a_i > 0$ and $d > 0$.
\item $h^d \sigma_2^m$ with $m \in \Z$ and either $d = 0$ $m \geq 0$ or $d \geq 1$.
\item $h^d \sigma_1^m \sigma_2^{-1}$ with $m \in \{ -1 ,-2, -3\}$ and $d > 0$.
\end{enumerate}

\end{thm}

\begin{thm}[Theorem 4.2 of~\cite{baldwin_heegaard_2008}]
For a 3-braid $\beta$, the contact invariant $c(T,\phi)$ of $\Sigma(\widehat{\beta})$ is non-vanishing if and only if $\widehat{\Theta}(\widehat{\beta}) $ is non-vanishing.
\end{thm}

The following statements summarize how the four invariants, $d_t$, $\psi$, $\widehat{\Theta}$, and $c(T,\phi)$, compare for 3-braids.

\begin{cor}
The following 3-braids have $s(\widehat{\beta}) = w(\widehat{\beta}) -2$, $\psi(\widehat{\beta}) \not = 0$, $\widehat{\Theta}(\widehat{\beta}) \not= 0 $ and $c(T,\phi) \not = 0$:
\begin{enumerate}
\item $h^d\sigma_1 \sigma_2^{-a_1} \cdots \sigma_1 \sigma_2^{-a_n}$ with $a_i \geq 0$ and some $a_i > 0$ and $d > 0$.
\item $h^d \sigma_2^m$ with $m \in \Z$ and either $d = 0$ $m \geq 0$, $d = 1$  $m \geq -4$, or $d > 1$.
\item $h^d \sigma_1^m \sigma_2^{-1}$ with $m \in \{ -1 ,-2, -3\}$ and $d > 0$.
\end{enumerate}
\end{cor}

\begin{cor}
The 3-braids $h\sigma_2^{m}$ with $m \leq -5$ have $\widehat{\Theta}(\widehat{\beta}) \not= 0 $ and $c(T,\phi) \not = 0$ but $\psi(\widehat{\beta}) = 0$ and $s(\widehat{\beta}) = w(\widehat{\beta})$.
\end{cor}

For the case of 3-braids the Heegaard Floer invariants completely classify the braids as right-veering and non right-veering depending on if the invariants vanish or not. However the comparison shows that the invariants defined from Khovanov homology may detect slightly more subtle information about the conjugacy class of $\beta$ as an element of $\Mod(D_3)$ because the invariants vanish/have non-maximal slope on the closures of $h\sigma_2^{m}$ with $m \leq -5$ and these braids are right-veering but contain large amounts of negative twisting inside a fixed circle in $D_3$ fixed by $\beta$.

\section{Computation of $s(\widehat{h^{-1}\sigma_2^4})$}\label{computation} 
The computation of the $s$ invariant of the braid closure of $h^{-1}\sigma_2^4$ makes use of Bar-Natan's cobordism category, which allows for a ``divide and conquer" approach to calculations. Recently Schuetz wrote a paper~\cite{schuetz_fast_2018} where he described using a ``divide and conquer" approach to compute the $s$ invariant of knots. While the specific ideas in his paper are not used in our calculation, his paper did prompt the author to consider using a ``divide and conquer" approach to compute $s(\widehat{h^{-1}\sigma_2^4})$.

We will make repeated use of the ideas in~\cite{bar-natan_fast_2007}, which allow us to build up the formal Bar-Natan complex crossing by crossing. We use Bar-Natan's delooping and cancellation lemmas (Lemmas 4.1 and 4.2 of~\cite{bar-natan_fast_2007}) to simplify the complex as much as possible before adding the next crossing. For the convenience of the reader, statements of these lemmas are included here. Additionally, the local relations in Bar-Natan's cobordism category are shown in Figure~\ref{Relations}.

\begin{lem}[Lemma~4.1~of~\cite{bar-natan_fast_2007}]
If an object $S$ in the Bar-Natan cobordism category contains a circle, then $S$ is isomorphic to $S' \{ 1\} \oplus S'\{-1\}$ where $S'$ is the object obtained from $S$ by removing the circle and the shifts are applied to the quantum grading.
\end{lem}
The isomorphisms are shown in Figure~\ref{Deloop}.

\begin{figure}

  \centering
    \includegraphics[width=0.8\textwidth]{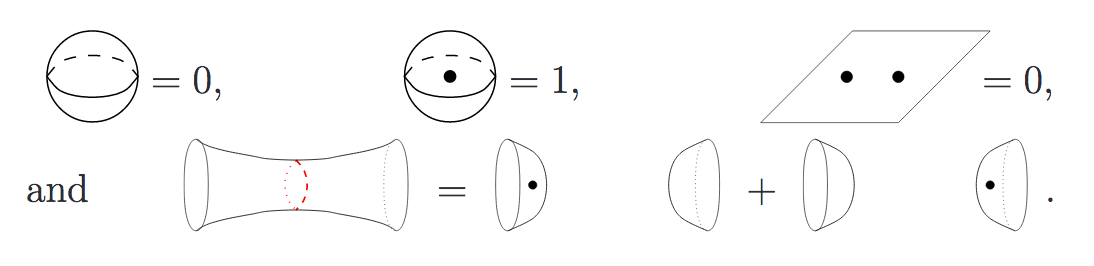}
      \caption{The local relations satisfied in Bar-Natan's cobordism category. Figure from~\cite{bar-natan_fast_2007}}\label{Relations}
\end{figure}
\begin{figure}

  \centering
    \includegraphics[width=0.6\textwidth]{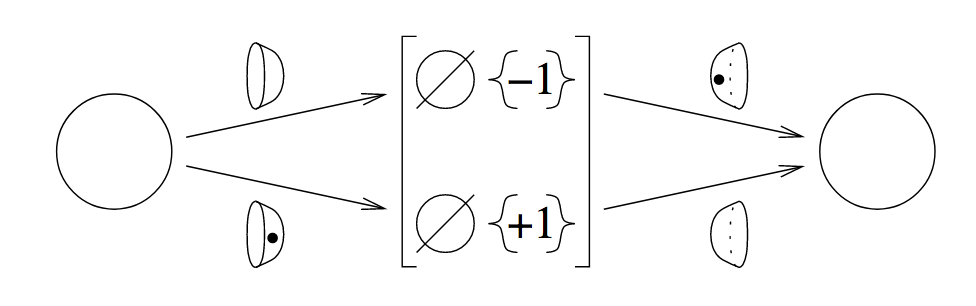}
      \caption{Delooping a circle to two shifted empty sets. Figure from~\cite{bar-natan_fast_2007}}\label{Deloop}
\end{figure}

\begin{lem}[Lemma~4.2 of~\cite{bar-natan_fast_2007} and Lemma~4.1 of \cite{baldwin_spectral_2011}]
If $\phi: B \to B'$ is an isomorphism, then the following chain complexes
 $$ \xymatrix@C=2cm{\cdots \ar[r] & A \ar[r]^{\begin{pmatrix}\alpha \\ \beta\end{pmatrix}} & B \oplus C \ar[r]^{\begin{pmatrix} \phi & \delta \\ \gamma & \epsilon \end{pmatrix}} & B' \oplus D \ar[r]^{\begin{pmatrix} \mu & \nu \end{pmatrix}} & E\ar[r] &\cdots}$$  is homotopy equivalent to the complex $$\xymatrix@C=2cm{\cdots \ar[r] & A \ar[r]^{\left(\beta\right)} &  C \ar[r]^{\left(\epsilon-\gamma\phi^{-1}\delta\right)} &  D \ar[r]^{\left(\nu\right)}& E\ar[r]&\cdots}$$

\end{lem}

With this background we are now ready to compute $s(\widehat{h^{-1}\sigma_2^4})$. A computer computation shows that the Khovanov homology of $\widehat{h^{-1}\sigma_2^4}$ with coefficients in $\mathbb{Q}$ in homological grading zero has dimension $6$ which is exactly the same as the dimension of the Lee homology of $\widehat{h^{-1}\sigma_2^4}$ in homological grading zero. So if you begin with the original Lee complex and use filtered chain homotopies to simplify it to a complex whose underlying generators are the generators of Khovanov homology, then there are no induced differentials entering or leaving homological grading 0. So in this homological grading, the filtration on Lee homology is determined by entirely by the grading on Khovanov homology.

The Khovanov homology of $\widehat{h^{-1}\sigma_2^4}$ in homological grading zero has dimension one in quantum grading $-3$, dimension three in quantum grading $-1$, and dimension two in quantum grading $1$. The part of the distinguished generator $\mathfrak{s}_o$ in quantum grading $-3$ is a cycle in the Khovanov homology of $\widehat{h^{-1}\sigma_2^4}$ so if it is non-zero in Khovanov homology then $s(\widehat{h^{-1}\sigma_2^4}) = -2$ because of the fact that the ranks of Khovanov homology and Lee homology are the same in homological grading 0.

To compute that the part of the distinguished generator $\mathfrak{s}_o$ in quantum grading $-3$ is non-zero in the Khovanov homology of the closure of $h^{-1}\sigma_2^4 = \sigma_1^{-1} \sigma_2^{-2} \sigma_1^{-1} \sigma_2^2$, first we split the braid into the words $\sigma_1^{-1} \sigma_2^{-2} \sigma_1^{-1}$ and $\sigma_2^2$ and compute a tangle complex for each. While computing and simplifying each tangle complex, one can keep track of which are in the image of the oriented resolution under chain homotopy. In this example, from the choice of where to cut the braid, the image is only the oriented resolution. In all of the tangle complexes, the homological and quantum grading of the objects are shown below the object as a pair $(i,j)$.

For the half of the braid $\sigma_2^2$ we start with the complex $A_1$ in Figure~\ref{A1} corresponding to $\sigma_2$ with a shift applied, which is the global shift of the braid $\sigma_1^{-1} \sigma_2^{-2} \sigma_1^{-1}\sigma_2^2$. To get a complex for $\sigma_2^2$ we take a complex $A_2$ in Figure~\ref{A2} for the second crossing $\sigma_2$ and tensor the two complexes to get the complex $A_3 = A_1 \otimes A_2$ in Figure~\ref{A3}. The object at the far right of $ A_3$ contains a trivial circle so it can be delooped to get the complex $A_4$ in Figure~\ref{A4}. The arrow at the bottom right of $A_4$ is an isomorphism so it can be eliminated to get the complex $B$ in Figure~\ref{B}. The only generator in $B$ that is in the image of the oriented resolution under the chain equivalence is the generator at the far left of the complex.

\begin{figure}

  \centering
    \includegraphics[width=0.6\textwidth]{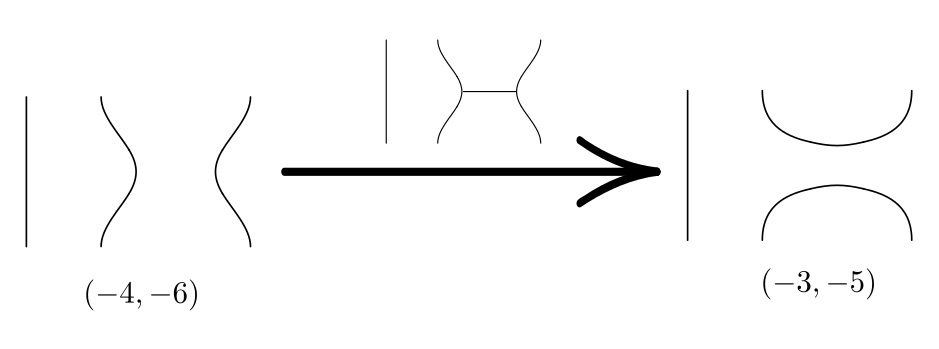}
      \caption{The complex $A_1$}\label{A1}
\end{figure}
\begin{figure}

  \centering
    \includegraphics[width=0.6\textwidth]{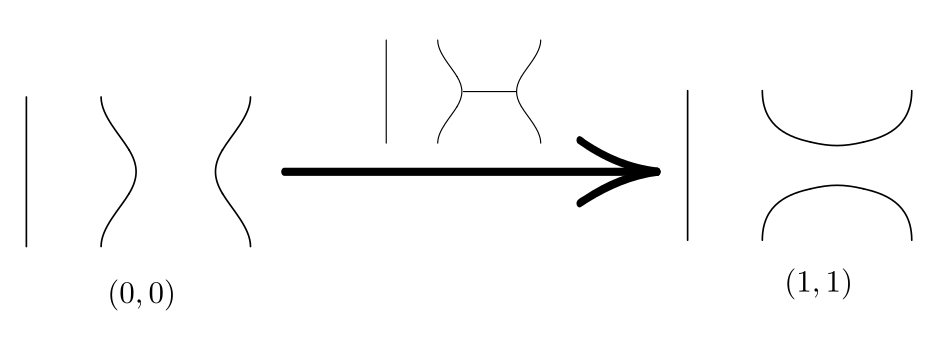}
      \caption{The complex $A_2$}\label{A2}
\end{figure}
\begin{figure}

  \centering
    \includegraphics[width=0.6\textwidth]{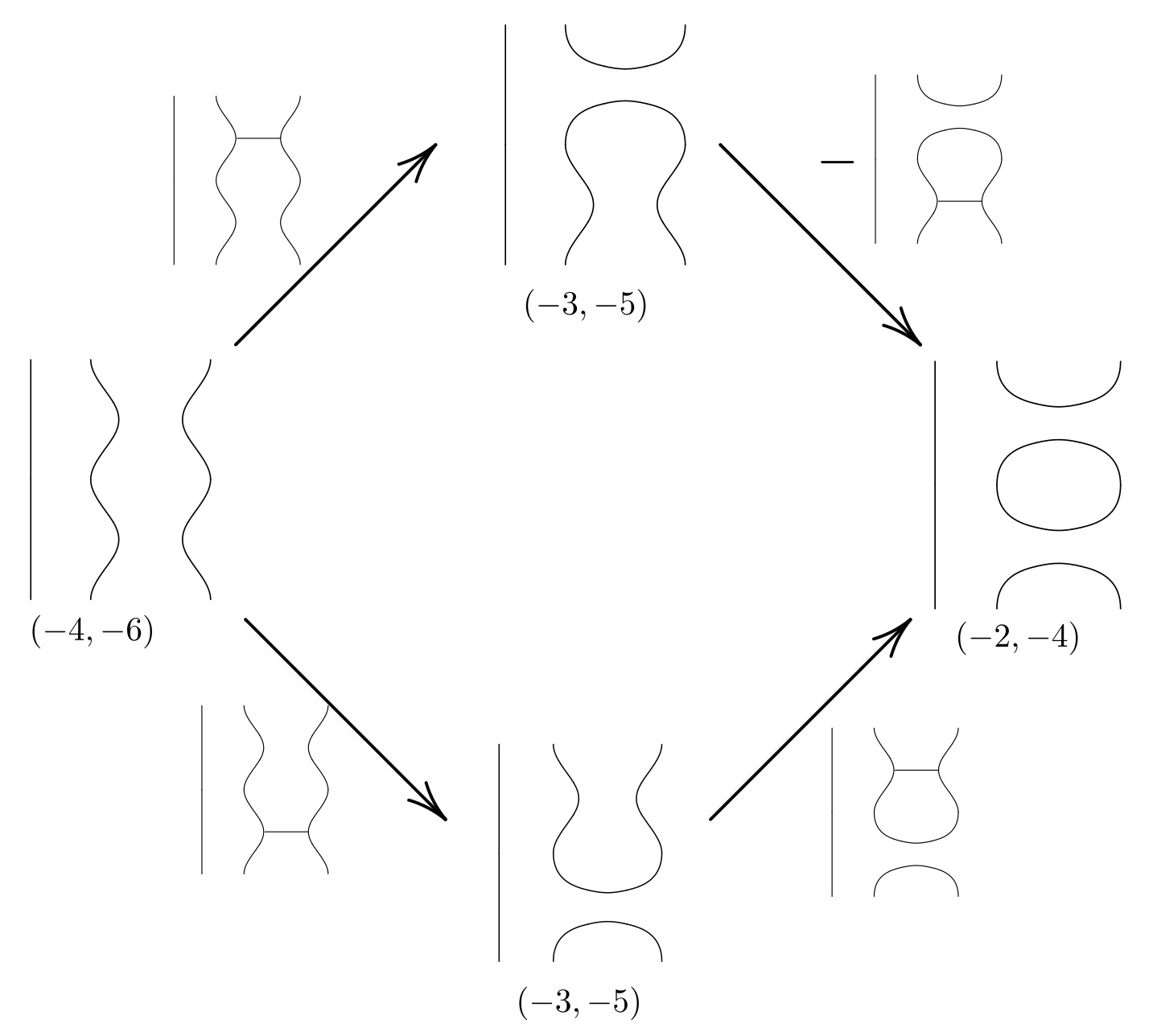}
      \caption{The complex $A_3=A_1\otimes A_2$}\label{A3}
\end{figure}
\begin{figure}

  \centering
    \includegraphics[width=0.9\textwidth]{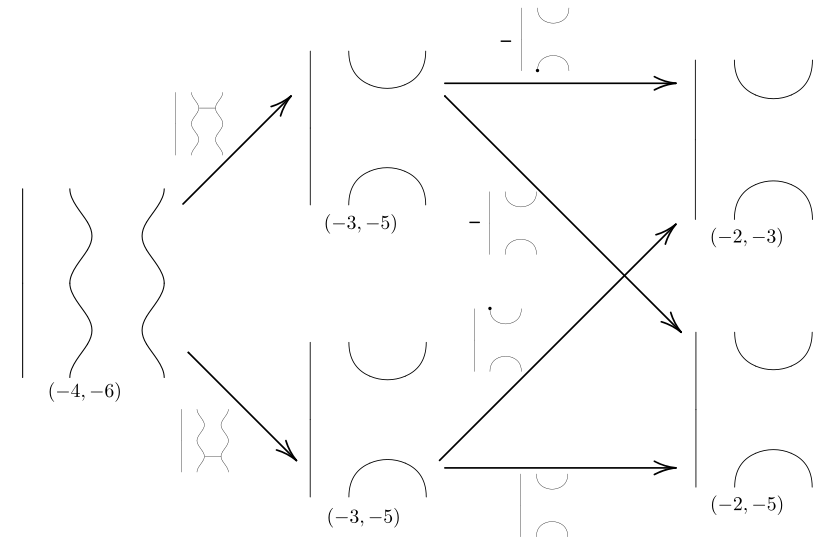}
      \caption{The complex $A_4$ from delooping $A_3$}\label{A4}
\end{figure}
\begin{figure}

  \centering
    \includegraphics[width=0.8\textwidth]{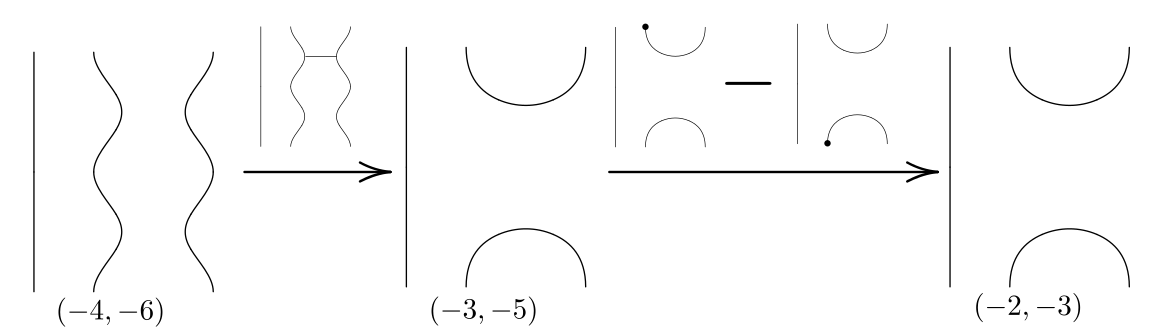}
      \caption{The complex $B$ from an elimination in $A_4$}\label{B}
\end{figure}

Now we compute the complex associated to the other half of the braid $\sigma_1^{-1} \sigma_2^{-2} \sigma_1^{-1}$. We start with the complex $C_1$ in Figure~\ref{C1} corresponding to $\sigma_2^{-1}$. Then we have the complex $C_2 = C_1 \otimes C_1$ in Figure~\ref{C2} which corresponds to $\sigma_2^{-2}$. The object at the far left of $ C_2$ contains a trivial circle so it can be delooped to get the complex $C_3$ in Figure~\ref{C3}. The arrow at the top left of $C_3$ is an isomorphism so it can be eliminated to get the complex $C_4$ in Figure~\ref{C4}. 

To get a complex representing $\sigma_2^{-2} \sigma_1^{-1}$ we take the complex $C_4$ and tensor it with the complex $C_5$ in Figure~{C5} to get the complex $C_6= C_4\otimes C_5$ in Figure~\ref{C6}. Then to represent the word $\sigma_1^{-1} \sigma_2^{-2} \sigma_1^{-1}$ we take $C_5 \otimes C_6$, which is shown as complex $D$ in Figures~\ref{ComplexGensD}~and~\ref{DiffsD}. The object $D_8$ in the complex $D$ has a closed circle so it is possible to deloop it and obtain the complex $E$ in Figures~\ref{GensE}~and~\ref{DiffsE}.

The map $E_4 \to E_8$ in the complex $E$ is an isomorphism so it can be eliminated. No other differentials change during this elimination because no other objects mapped into $E_8$. The map $E_8' \to E_{11}$ in the complex is an isomorphism as well so it can also be eliminated. The differential of $E_7$ changes with this elimination. After both eliminations we get a new complex $E'$ whose objects are $\{E_1 , E_2 , E_3, E_5 , E_6 , E_7, E_9 , E_ {10}, E_{12}\}$ and whose differentials are shown in Figure~\ref{DiffsEprime}.

Now we are ready to combine the two halves of the braid $\sigma_{1}^{1} \sigma_2^{-2} \sigma_1^{-1}$ and $\sigma_2^2$ by tensoring the complexes $E'$ and $B$. Because we are considering if a specific cycle in homological grading $0$ is non-vanishing in homology, we only need to consider objects in homological gradings $-1$ and $0$ of the complex $E' \otimes B$. The generators of this complex in these homological degrees are shown in Figure~\ref{FinalComplexGens} and the differentials from grading $-1$ to grading $0$ are shown in Figure~\ref{FinalComplexDiffs}. For the sake of clarity we number the objects of the complex $B$ from $1$ to $3$ by increasing homological grading and label the objects in $E' \otimes B$ as the tensor of an object from each complex.

We can close the generators of $E' \otimes B$ into planar unlinks and consider specifically the generators in quantum grading $-3$. These generators are enumerated in Figure~\ref{ClosedResolutions} where a list of $+$ and $-$ signs means labeling the circles with these signs in the order shown by the numbering of the circles. The cycle which is the part of $\mathfrak{s}_o$ in this quantum grading is $-v_{10} + v_{11} - v_{12}$.

For these generators in Figure~\ref{ClosedResolutions}, the differential from homological grading $-1$ to homological grading $0$ is shown in Table~\ref{DiffTable}. A computer computation then shows that the cycle $ -v_{10} + v_{11} - v_{12}$ is not in the image of the differential and so it is non-zero in the Khovanov homology of the link and so then $s(\widehat{h^{-1}\sigma_2^4}) = -2$.

\begin{figure}

  \centering
    \includegraphics[width=0.6\textwidth]{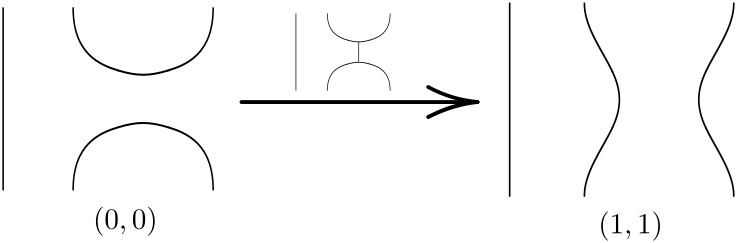}
      \caption{The complex $C_1$ }\label{C1}
\end{figure} 

\begin{figure}

  \centering
    \includegraphics[width=0.65\textwidth]{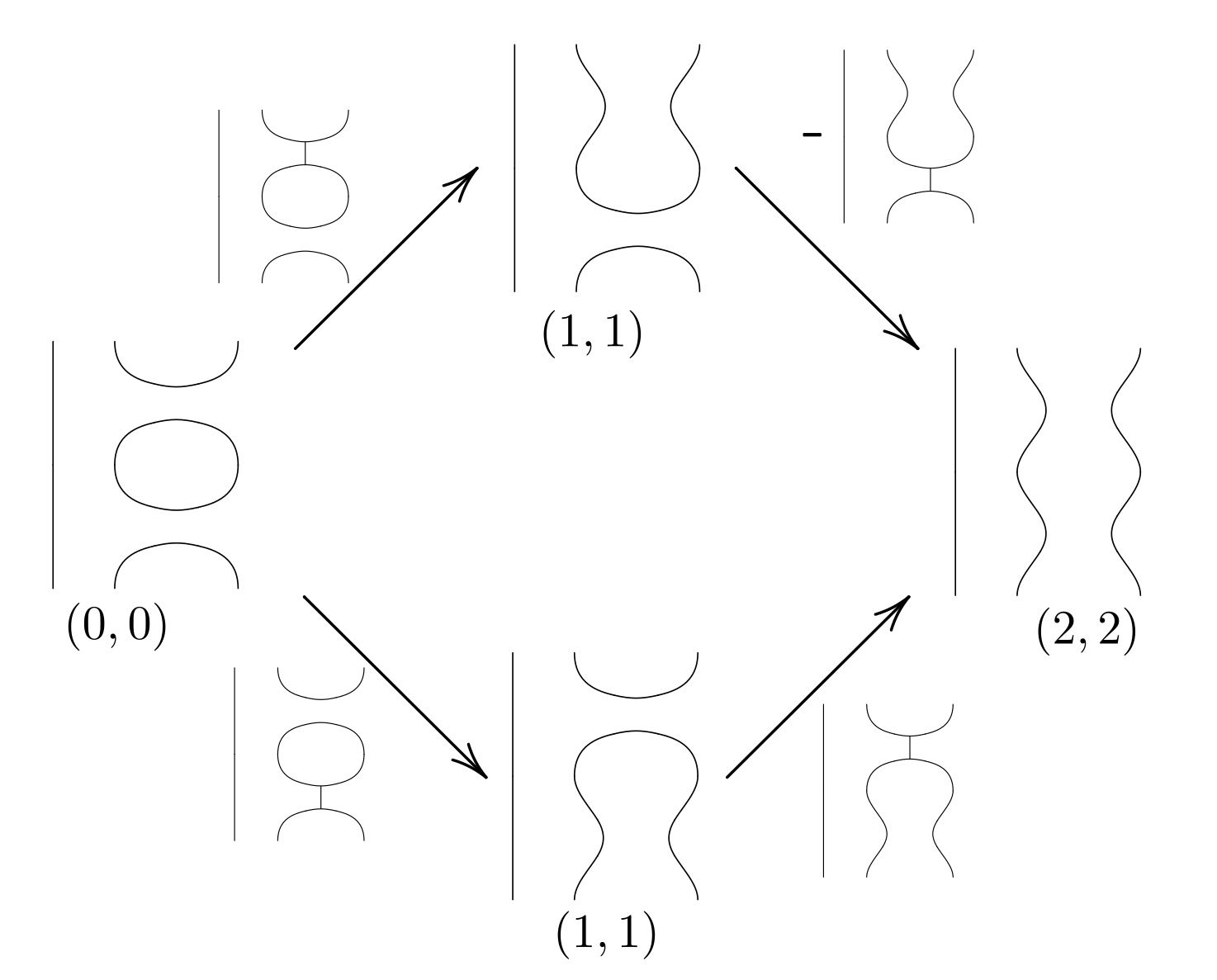}
      \caption{The complex $C_2 = C_1 \otimes C_1$ }\label{C2}
\end{figure} 

\begin{figure}

  \centering
    \includegraphics[width=0.55\textwidth]{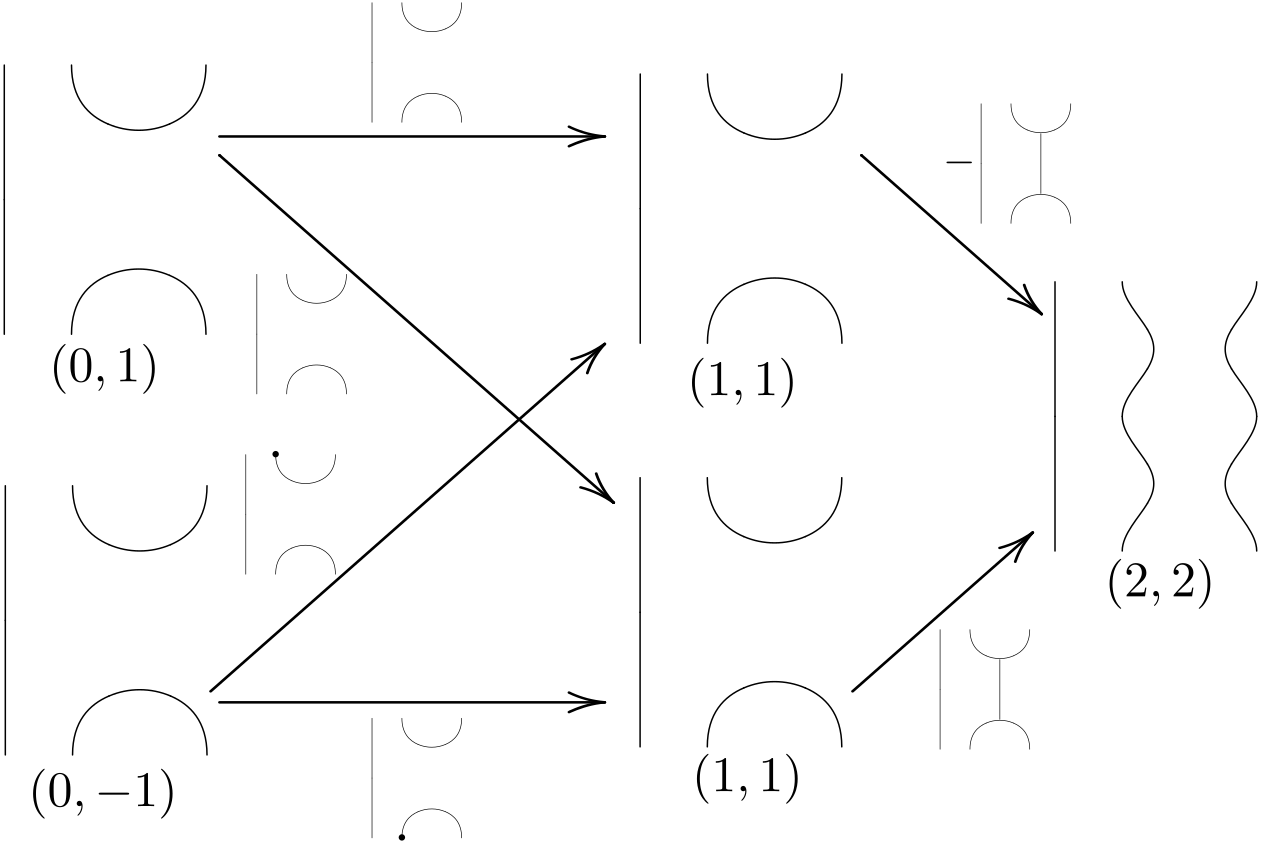}
      \caption{The complex $C_3$ from delooping $C_2$ }\label{C3}
\end{figure} 

\begin{figure}

  \centering
    \includegraphics[width=0.65\textwidth]{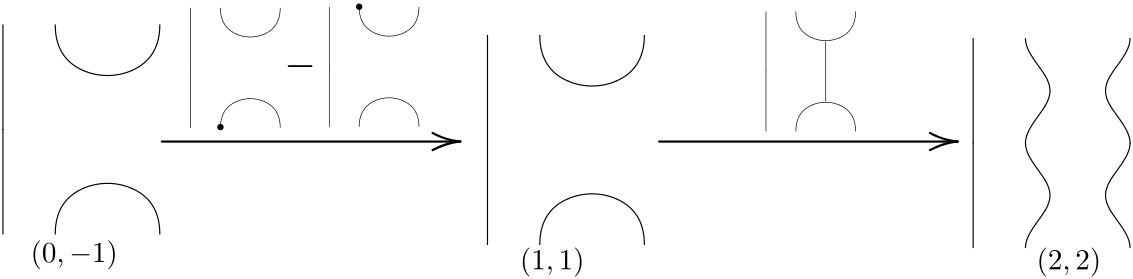}
      \caption{The complex $C_4$ from an elimination in $C_3$ }\label{C4}
\end{figure} 

\begin{figure}

  \centering
    \includegraphics[width=0.45\textwidth]{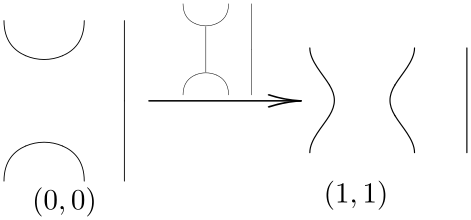}
      \caption{The complex $C_5$ representing $\sigma_1^{-1}$ }\label{C5}
\end{figure} 

\begin{figure}

  \centering
    \includegraphics[width=.9\textwidth]{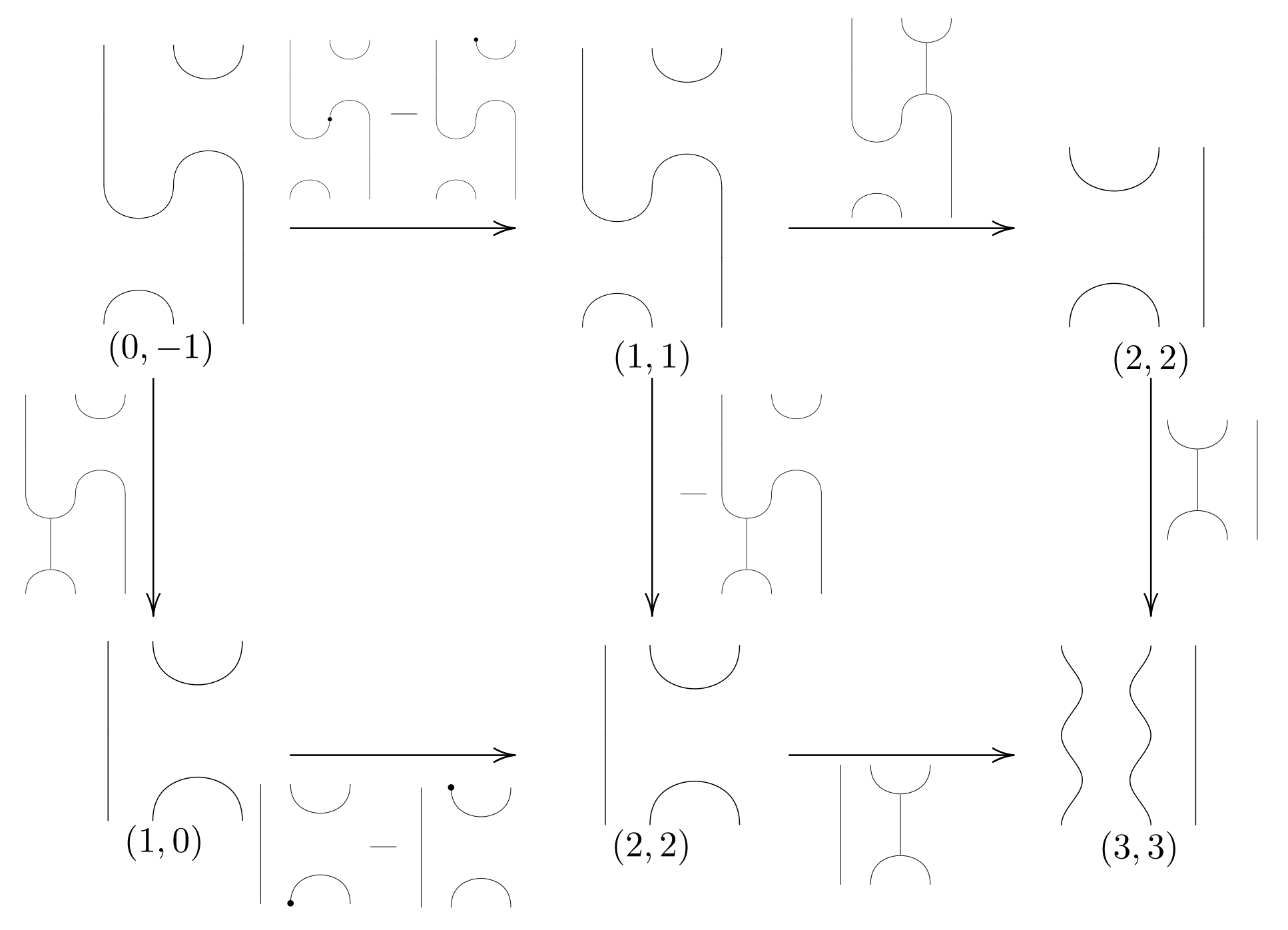}
      \caption{The complex $C_6 = C_4 \otimes C_5$ }\label{C6}
\end{figure} 

\begin{figure}

  \centering
    \includegraphics[width=0.6\textwidth]{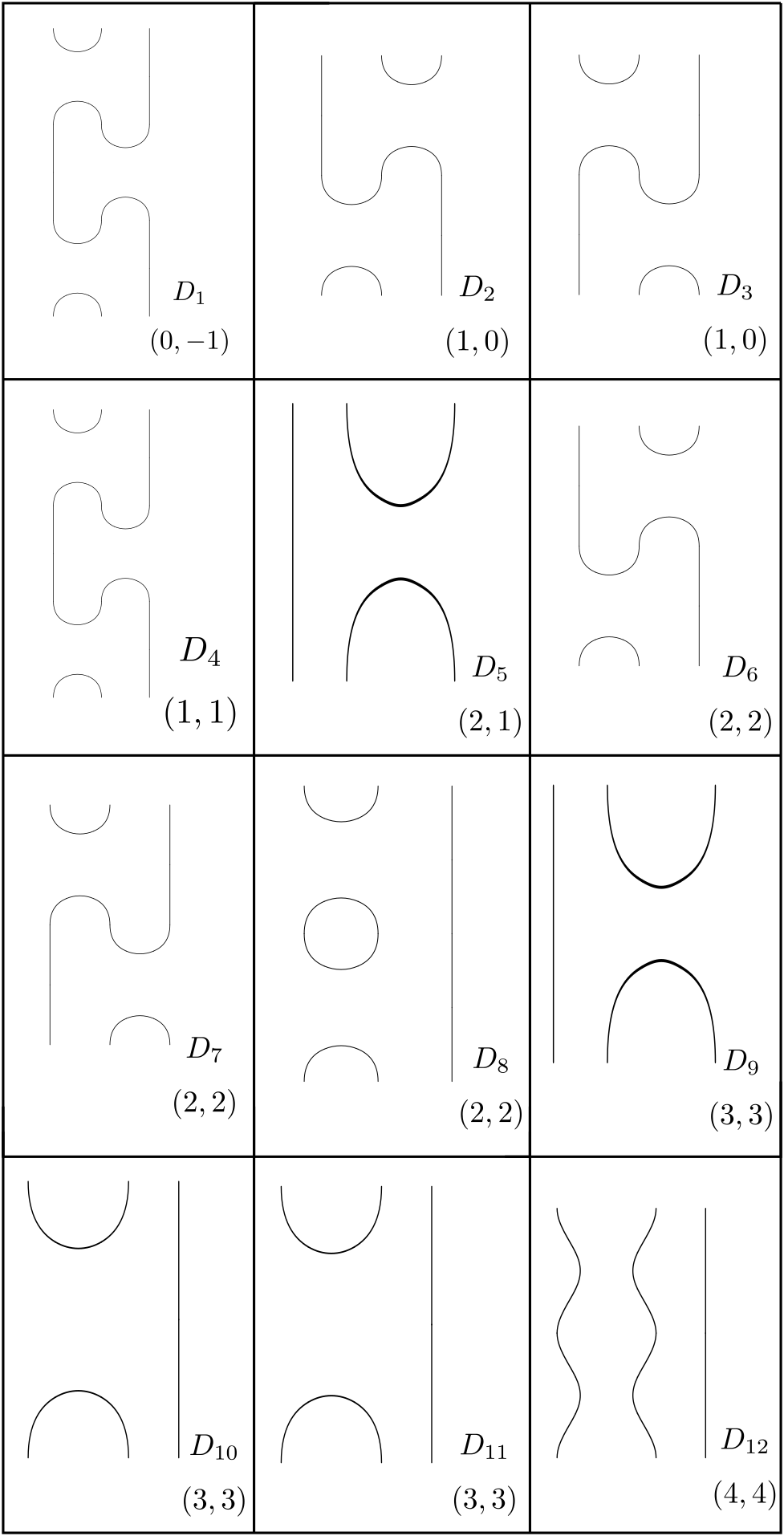}
      \caption{The generators of the complex $D= C_5 \otimes C_6$ }\label{ComplexGensD}
\end{figure} 

\begin{figure}

  \centering
    \includegraphics[width=\textwidth]{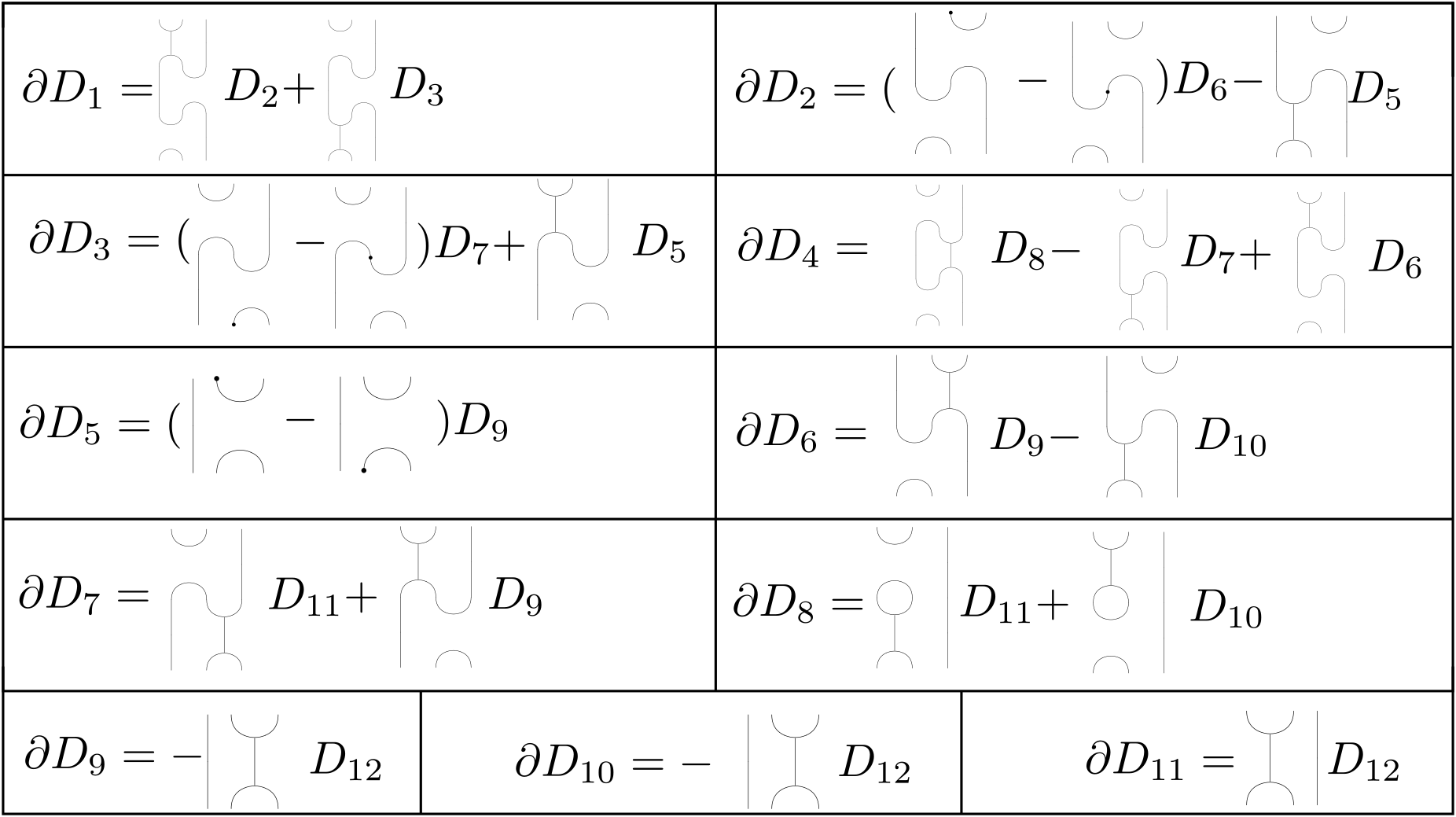}
      \caption{The differentials of the complex $D= C_5 \otimes C_6$ }\label{DiffsD}
\end{figure} 

\begin{figure}

  \centering
    \includegraphics[width=0.6\textwidth]{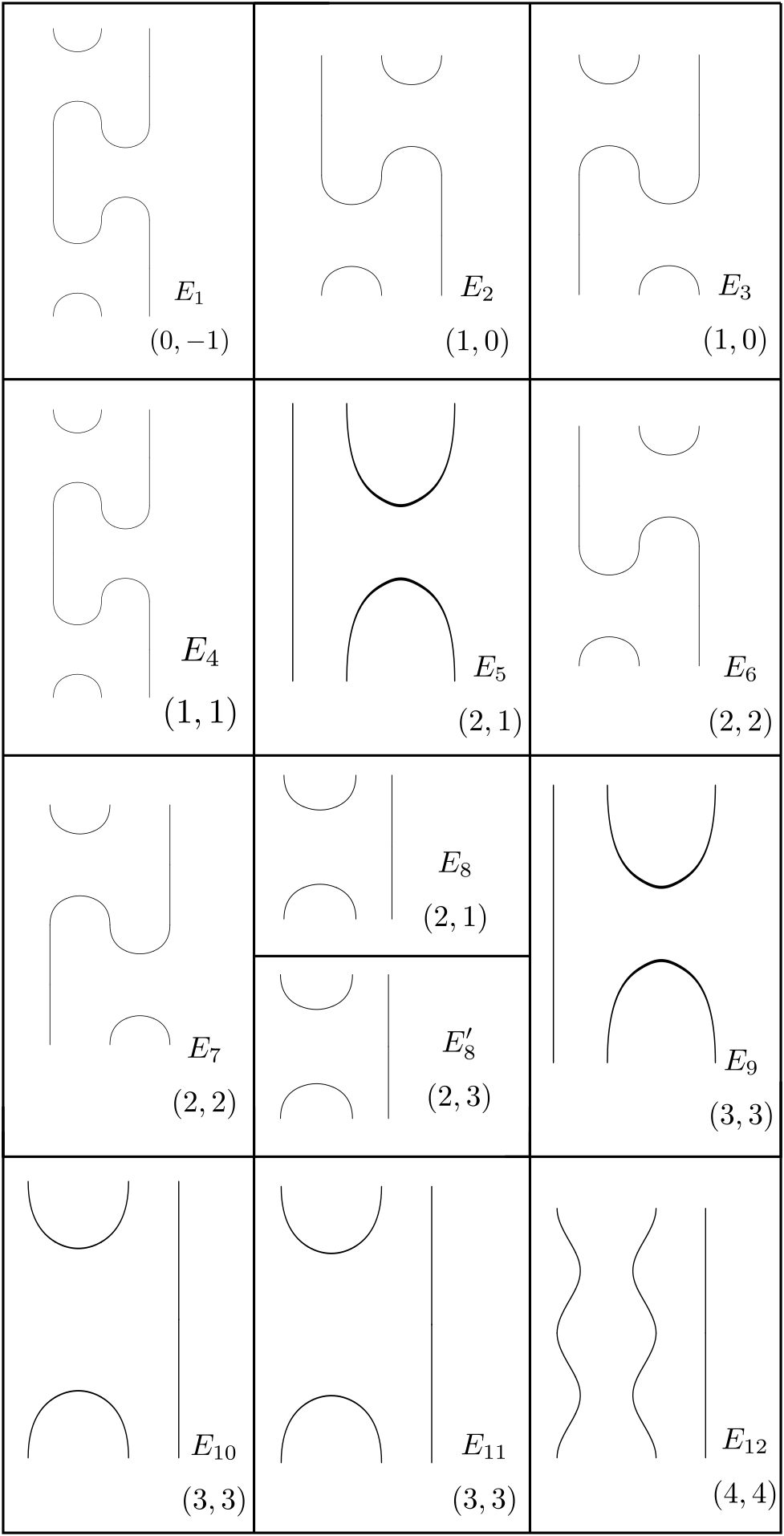}
      \caption{The generators of the complex $E$ obtained by delooping the complex $D$ }\label{GensE}
\end{figure} 

\begin{figure}

  \centering
    \includegraphics[width=\textwidth]{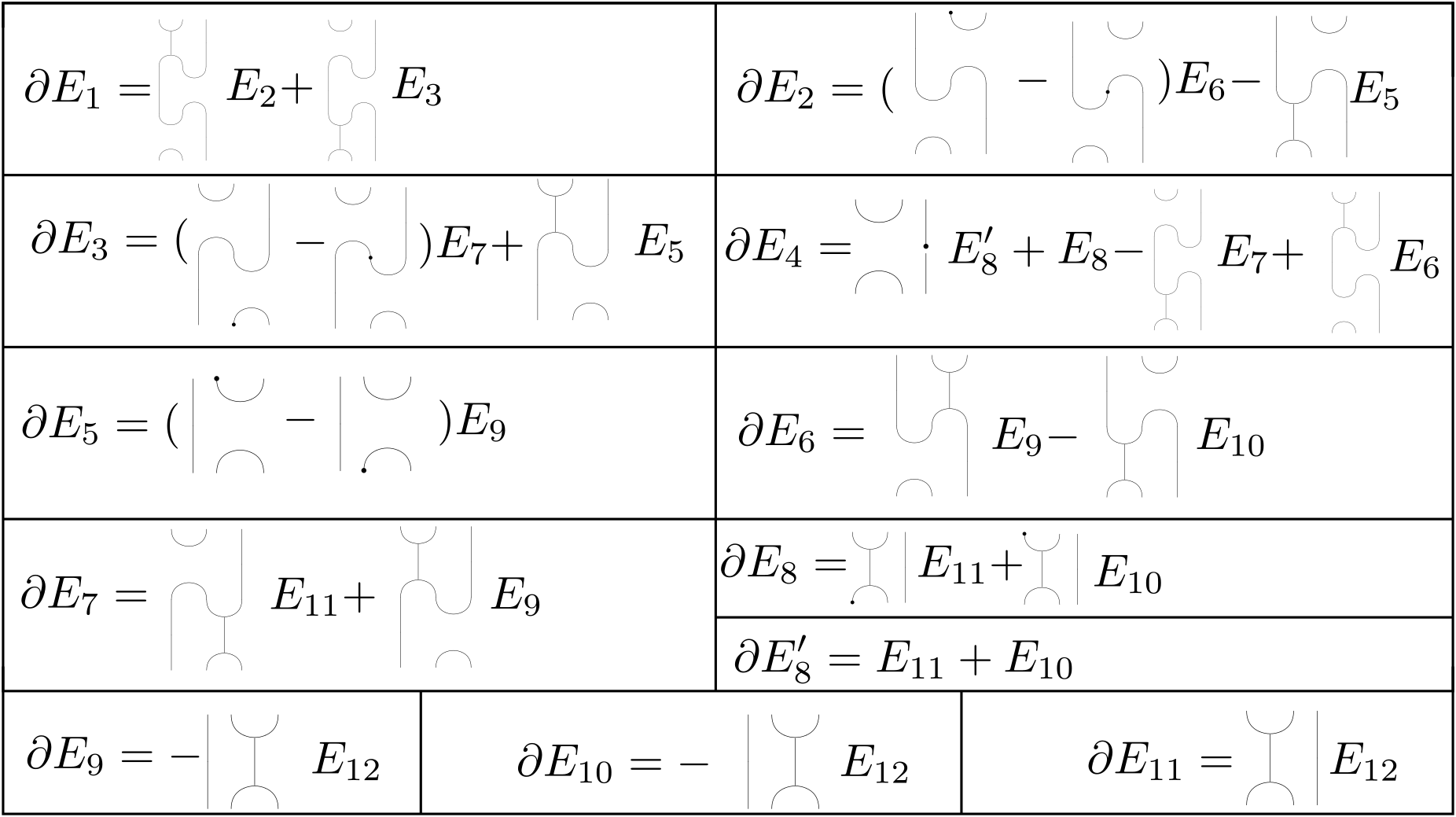}
      \caption{The differentials of the complex $E$ }\label{DiffsE}
\end{figure} 

\begin{figure}

  \centering
    \includegraphics[width=\textwidth]{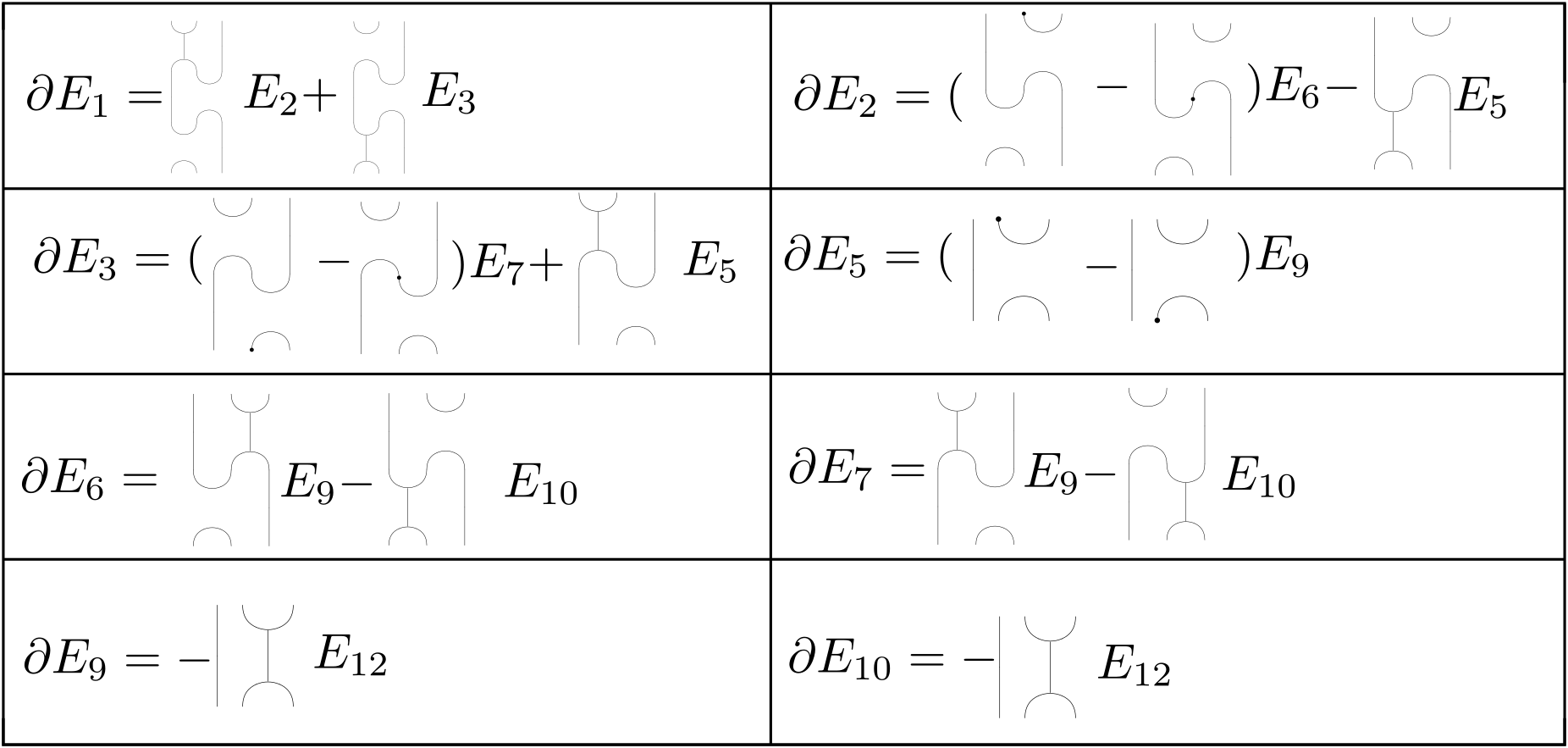}
      \caption{The differentials of the complex $E'$ obtained by eliminations of $E$ }\label{DiffsEprime}
\end{figure} 

\begin{figure}

  \centering
    \includegraphics[width=\textwidth]{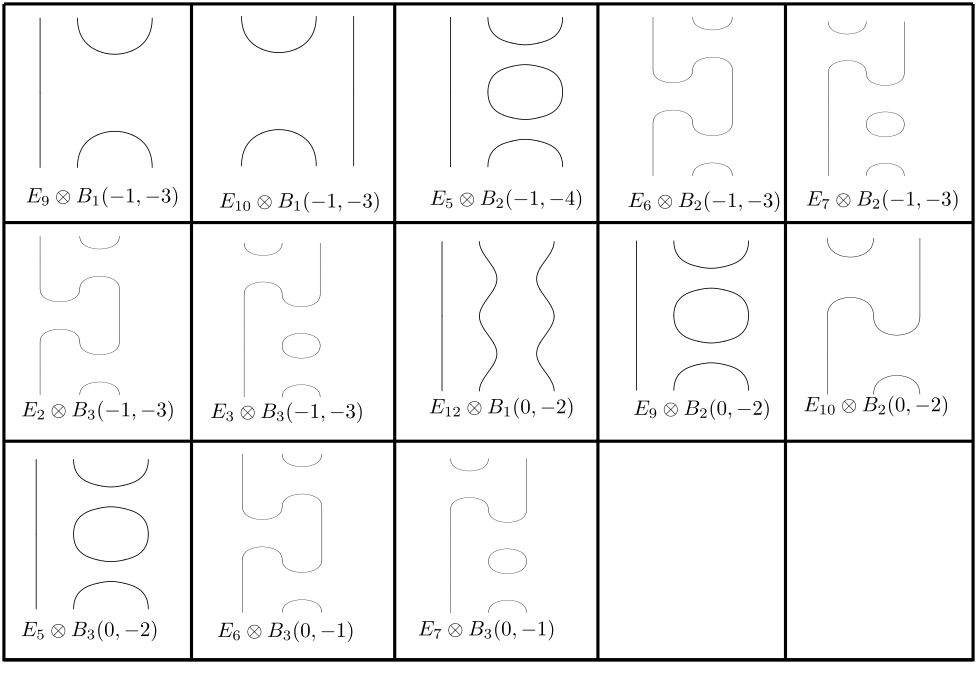}
      \caption{The generators of $E' \otimes B$ in homological gradings $-1$ and $0$ }\label{FinalComplexGens}
\end{figure} 

\begin{figure}

  \centering
    \includegraphics[width=\textwidth]{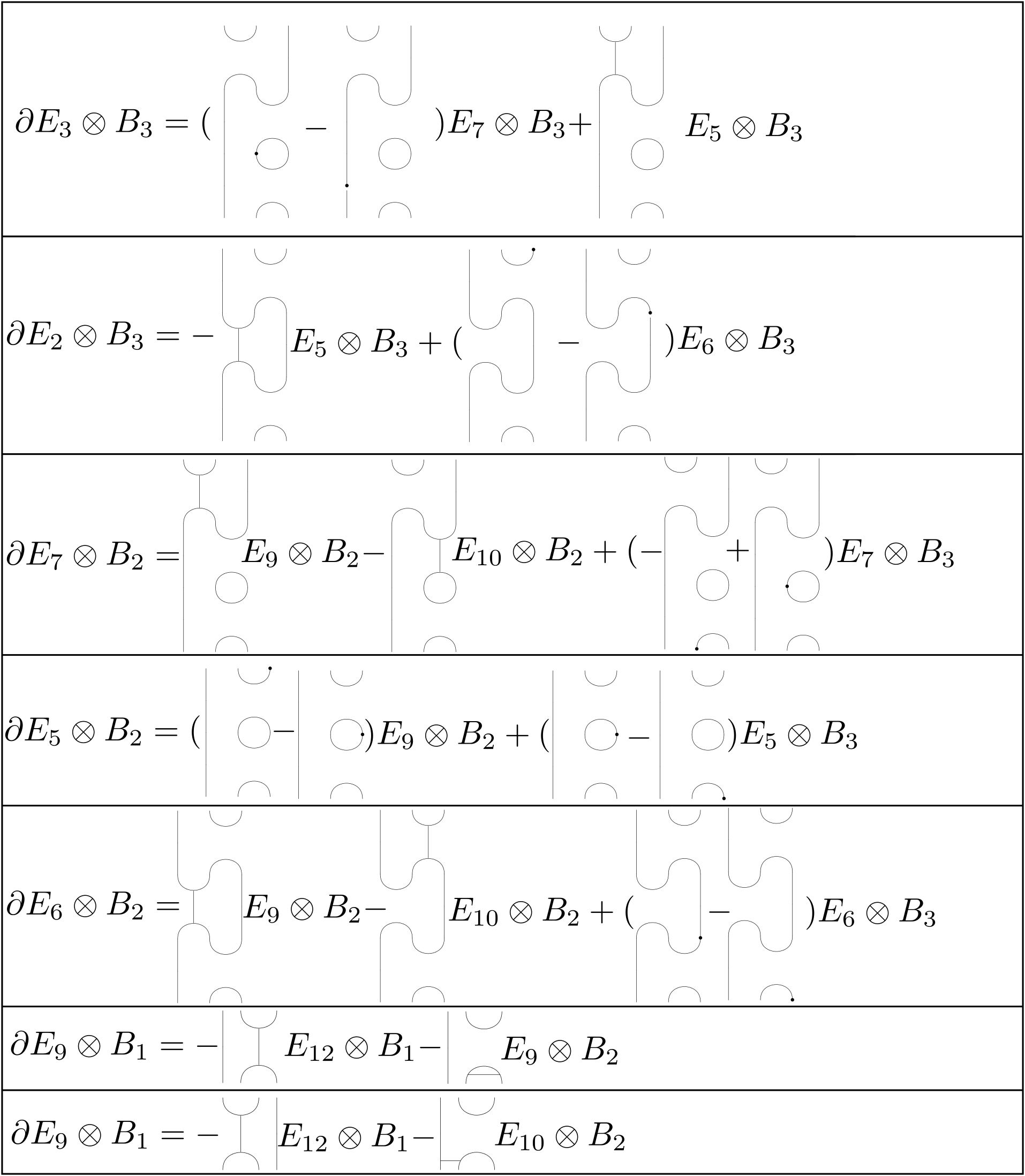}
      \caption{The differentials of $E' \otimes B$ out of homological gradings $-1$ }\label{FinalComplexDiffs}
\end{figure} 

  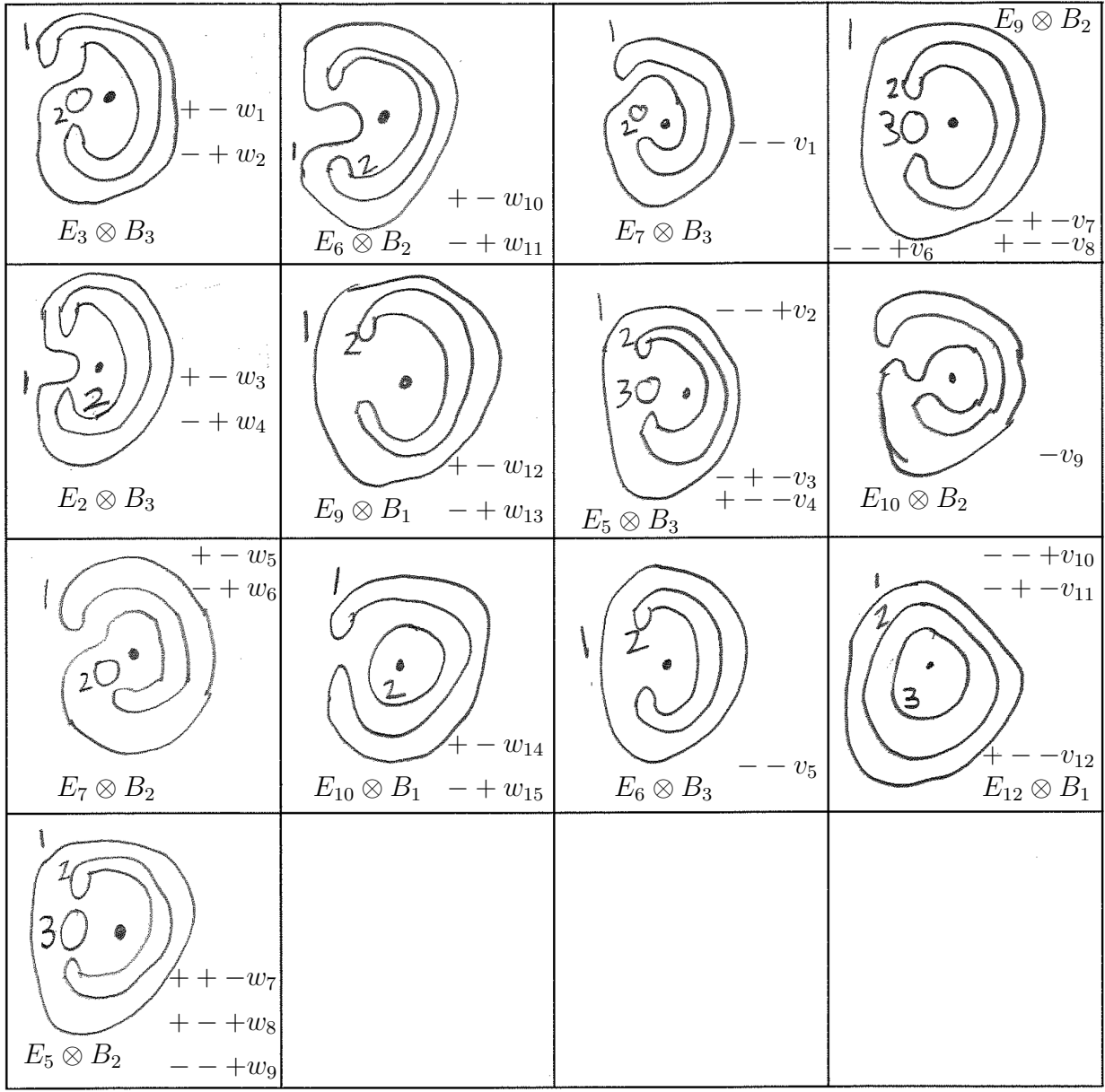
\begin{figure}

  \centering
    \begin{overpic}[width=\textwidth]{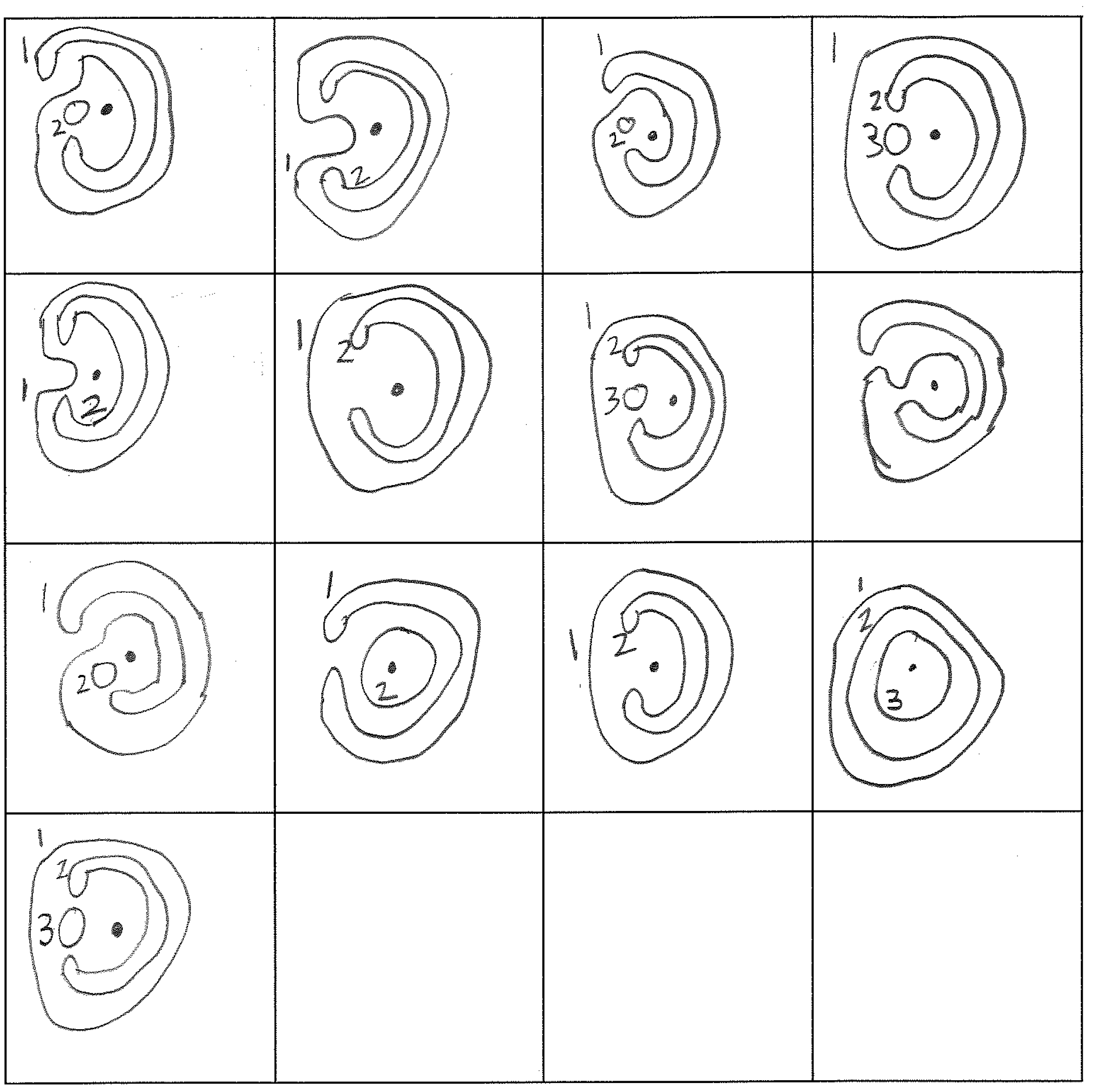}
    \put(5,77){$E_3 \otimes B_3$}
    \put(16,88){$+-w_1$}
     \put(16,84){$-+w_2$}
     
      \put(5,53){$E_2 \otimes B_3$}
      \put(16,64){$+-w_3$}
     \put(16,60){$-+w_4$}

  \put(5,27){{$E_7 \otimes B_2$}}
  \put(17,48){$+-w_5$}
     \put(17,45){$-+w_6$}
  
  \put(2,3){$E_5 \otimes B_2$}
  \put(15,10){$++-w_7$}
     \put(15,6){$+-+w_8$}
      \put(15,2){$--+w_9$}

     \put(28,76){$E_6\otimes B_2$}
    \put(40,80){$+-w_{10}$}
     \put(40,76){$-+w_{11}$}
     
 \put(28,52){$E_9\otimes B_1$}
     \put(40,56){$+-w_{12}$}
     \put(40,52){$-+w_{13}$}
     
  \put(28,27){$E_{10}\otimes B_1$}
      \put(40,31){$+-w_{14}$}
     \put(40,27){$-+w_{15}$}
     
    \put(55,77){$E_{7}\otimes B_3$}
     \put(66,85){$--v_{1}$}
     
      \put(52,51){$E_{5}\otimes B_3$}
           \put(64,70){$--+v_{2}$}
           \put(64,55){$-+-v_{3}$}
           \put(64,53){$+--v_{4}$}
           
         \put(55,27){$E_{6}\otimes B_3$}
              \put(66,29){$--v_{5}$}
              
         \put(89,96){$E_{9}\otimes B_2$}     
           \put(74.5,75.5){$--+v_{6}$}
           \put(89,78){$-+-v_{7}$}
           \put(89,76){$+--v_{8}$}
           
         \put(77,53){$E_{10}\otimes B_2$} 
          \put(93,57){$-v_{9}$}
          
          \put(88,27){$E_{12}\otimes B_1$}             
           \put(88,48){$--+v_{10}$}
           \put(88,45){$-+-v_{11}$}
           \put(88,30){$+--v_{12}$}
          \end{overpic}
      \caption{The generators in quantum grading $-3$ after closing off the tangles}\label{ClosedResolutions}
      \end{figure}
      
      \begin{table}[]
  \begin{center}
    \begin{tabular}{|c|c|}
      \hline
      $\partial(w_1) = -v_1 + v_3 + v_4$ & $\partial(w_9) = - v_7 + v_6 - v_2 + v_3$ \\
      \hline
      $\partial(w_2) = v_1 + v_2$ & $\partial(w_{10}) = v_6 + v_8 - v_9 + v_5$ \\
      \hline
      $\partial(w_3) = -v_2 - v_4 - v_5$ & $\partial(w_{11}) = v_7 - v_9 - v_5$ \\
      \hline
      $\partial(w_4) = -v_3 + v_5$&$\partial(w_{12}) = - v_{12} - v_8$ \\
      \hline
      $\partial(w_5) = v_7 + v_8 - v_9 - v_1 $ & $\partial(w_{13}) = - v_{10} - v_{11} - v_6 - v_7$ \\
      \hline
        $\partial(w_6) =v_6 - v_9 + v_1$ & $\partial(w_{14}) = - v_{10} - v_{11} - v_6 - v_7$ \\
      \hline
       $\partial(w_7) = v_8 - v_4 $ & $\partial(w_{15} )= - v_{10} - v_9$ \\
      \hline
             $\partial(w_8) = -v_8 + v_4 $ &  \\
      \hline
      
    \end{tabular}
        \caption{The differentials from homological grading $-1$ to homological grading $0$ for the generators in Figure~\ref{ClosedResolutions}}
    \label{DiffTable}
  \end{center}
\end{table}

\clearpage
\bibliography{3BraidPaper}

\end{document}